\documentclass[11pt]{amsart}

\usepackage{amscd}
\usepackage{xypic}  
\usepackage{amssymb}
\usepackage{amsthm}
\usepackage{epsfig}
\usepackage{paralist}

\usepackage{tikz}

\usepackage[all,cmtip]{xy}
\usepackage{amsmath}
\usepackage{color}

\newtheorem{thm}{Theorem}[section]  

\newtheorem{lemma}[thm]{Lemma}

\newtheorem{proposition}[thm]{Proposition}

\newtheorem{corollary}[thm]{Corollary}

\newtheorem{claim}{Claim}[thm]

\theoremstyle{definition}

\newtheorem{remark}[thm]{Remark}

\newtheorem{question}[thm]{Question}

  \newtheorem{definition-remark}[thm]{Definition-Remark}
 \newtheorem{definition}[thm]{Definition}

\newtheorem*{example1*}{Example: The case $\phi_1=1$}

\newtheorem*{example2*}{Example: The case $\phi_1=2$}

\newtheorem*{example3*}{Example: The case $\phi_1=3$}

\def\ker{\operatorname{ker}}

\def\min{\operatorname{min}}

\def\max{\operatorname{max}}

\def\c1{\operatorname{c_1}}
\def\c2{\operatorname{c_2}}

\def\Sym{\operatorname{Sym}}

\def\ZZ{{\mathbb Z}}

\def\QQ{{\mathbb Q}}
\def\PP{{\mathbb P}}

\def\A{{\mathcal A}}
\def\B{{\mathcal B}}
\def\D{{\mathcal D}}
\def\DD{{\mathbb D}}

\def\L{{\mathcal L}}
\def\M{{\mathcal M}}
\def\N{{\mathcal N}}
\def\O{{\mathcal O}}
\def\I{{\mathcal J}}

\def\E{{\mathcal E}}
\def\T{{\mathcal T}}

\def\e{\mathfrak{e}}
\def\s{\mathfrak{s}}
\def\f{\mathfrak{f}}
\def\c{\mathfrak{c}}

\def\x{\times}                   
\def\cong{\simeq}

\def\+{\oplus}               
\def\*{\otimes}                  

\def\Bl{\operatorname{Bl}}

\def\Shom{\operatorname{ \mathfrak{h}\mathfrak{o}\mathfrak{m} }}
\def\Shext{\operatorname{ \mathfrak{e}\mathfrak{x}\mathfrak{t} }}

\def\Pic{\operatorname{Pic}}

\def\Num{\operatorname{Num}}

\def\Bl{\operatorname{Bl}}

\begin{document}

\title{On moduli spaces of polarized Enriques surfaces}

\author{Andreas Leopold Knutsen}
\address{Andreas Leopold Knutsen, Department of Mathematics, University of Bergen, Postboks 7800,
5020 Bergen, Norway}
\email{andreas.knutsen@math.uib.no}

\date{\today}


\begin{abstract}  
We prove that, for any $g \geq 2$, the {\'e}tale double cover $\rho_g:\E_{g} \to
\widehat{\E}_{g}$ from the moduli space $\E_{g}$ of complex polarized genus $g$ Enriques surfaces to the moduli space $\widehat{\E}_{g}$ of numerically polarized genus $g$ Enriques surfaces is disconnected precisely over irreducible components of $\widehat{\E}_{g}$ parametrizing $2$-divisible classes, answering a question of Gritsenko and Hulek \cite{GrHu}. We characterize all irreducible components of $\E_{g}$ in terms of a new invariant of line bundles on Enriques surfaces that generalizes the $\phi$-invariant introduced by Cossec \cite{cos2}. In particular, we get a one-to-one correspondence between the irreducible components of $\E_g$ and $11$-tuples of integers satisfying particular conditions. This makes it possible, in principle, to list all irreducible components of $\E_g$ for each $g \geq 2$.
\end{abstract}

\maketitle

\section{Introduction}

For any integer $g \geq 2$, let 
$\E_{g}$ (resp., $\widehat{\E}_{g}$) denote the moduli
space of complex polarized (resp. numerically polarized) 
Enriques surfaces $(S,L)$ (resp. $(S,[L])$) 
of {\it (sectional) genus} $g$, that is, such that 
$L^2=2g-2$. (Thus, $g$ is the arithmetic genus of all curves in the linear system
$|L|$.) The moduli spaces $\E_{g}$ exist as quasi-projective varieties by Viehweg's theory, cf. \cite[Thm. 1.13]{Vie}. The moduli spaces $\widehat{\E}_{g}$ exist by results of Gritsenko and Hulek \cite{GrHu}; more precisely,
for each orbit  $\mathfrak{h}$ of the
action of the orthogonal group in the {\it Enriques lattice} $\mathfrak{N}:=U \+ E_8(-1)$ (see  \cite[Lemma VIII.15.1]{BHPV}),
there is an irreducible moduli space $\M^a_{\tiny{\mbox{En}},\mathfrak{h}}$ parametrizing isomorphism classes of numerically polarized Enriques surfaces $(S,[L])$ with $[L]$ in the orbit $ \mathfrak{h} \subset \mathfrak{N} \cong \Num S$. 
The space $\widehat{\E}_{g}$ in our notation is thus  the union of all $\M^a_{\tiny{\mbox{En}},\mathfrak{h}}$ where $\mathfrak{h}$ varies over all orbits 
with $\mathfrak{h}^2=2g-2$. It follows from \cite[Prop. 4.1]{GrHu} that there is an {\'e}tale double cover $\rho_g: \E_{g} \to
\widehat{\E}_{g}$ identifying 
$(S,L)$ and $(S,L+K_S)$.
Note that in general the spaces $\E_{g}$ and $\widehat{\E}_{g}$ have many  irreducible components.

In this paper we answer the following fundamental questions:
\begin{itemize}
\item[(1)] Given an irreducible component of $\widehat{\E}_{g}$, is 
its inverse image by $\rho_g$  irreducible or not 
(cf. \cite[Question 4.2]{GrHu})? 
\item[(2)] How can one determine {\it all} the irreducible components of $\E_{g}$?

\end{itemize}

Regarding the first question, for each irreducible component $\widehat{\E}'$ of $\widehat{\E}_{g}$ either $\rho_g^{-1} \widehat{\E}'$ is  irreducible or it consists of two disjoint components, according to whether $(S,L)$ and $(S,L+K_S)$ lie in the same component of $\E_{g}$ or not for $(S,[L]) \in \widehat{\E}'$. We will prove:

\begin{thm} \label{mainthm1}
  Let $\widehat{\E}'$  be an irreducible component of $\widehat{\E}_{g}$.
Then $\rho_g^{-1} (\widehat{\E}')$ is reducible if and only if $\widehat{\E}'$ parametrizes pairs $(S,[L])$ such that $[L]$ is $2$-divisible in $\Num S$.
\end{thm}

We remark that a much weaker version of this theorem was obtained in \cite[Cor. 1.5]{cdgk}, with a completely different approach. 

Regarding question (2) above, one can start by 
fixing another fundamental invariant in addition to the genus, namely the {\it $\phi$-invariant}
\begin{equation} \label{eq:defphi}
\phi(L):=\min \left\{E \cdot L \; | \; E^2=0, E > 0\right\} \in \ZZ_+,
\end{equation}
introduced by Cossec \cite{cos2}, which has interesting geometrical interpretations, cf., e.g., \cite{cd, KLpn, kn-man,Sz}. Then one may, as in \cite{cdgk}, consider the moduli spaces $\E_{g,\phi}$ and $\widehat{\E}_{g,\phi}$ parametrizing pairs with
$L^2=2g-2$ and $\phi(L)=\phi$, which in general still have many different irreducible components. Also recall that not all possible pairs $(g,\phi)$ occur; for instance it is known that 
$\phi^2 \leqslant 2g-2$
by \cite[Cor. 2.7.1]{cd}, and that there  are no cases satisfying $\phi^2 < 2g-2< \phi^2 +\phi-2$ by \cite[Prop. 1.4]{KL1}, but a complete classification of all possible pairs $(g,\phi)$ is still missing.

Some irreducibility results have been known in low genus for a while; for instance $\E_{3,2}$, $\E_{4,2}$ and $\E_{6,3}$ are irreducible, see
\cite{Ca}, \cite[\S 3]{dol} and \cite{Ve}. 
In \cite{cdgk} all irreducible components
of $\E_{g,\phi}$ were determined for $\phi \leq 4$ or $g \leq 20$ and described in terms of decompositions of the line bundles they parametrize into effective, primitive isotropic decompositions (that is, into effective classes of square zero that are indivisible in $\Num S$),
cf. \S  \ref{sec:iso} below.
As a sample, which was classically known, $\E_{5,2}$ has three irreducible components, denoted by $\E_{5,2}^{(I)}$, $\E_{5,2}^{(II)^+}$ and $\E_{5,2}^{(II)^-}$ in \cite{cdgk2}, corresponding to the following decompositions of $L$ into effective, primitive  isotropic classes:
\begin{eqnarray*}
\E_{5,2}^{(I)}  \hspace{0.3cm}    & \hspace{-0.5cm}L \sim 2E_1 +E_{1,2},    &  \hspace{0.2cm}E_1 \cdot E_{1,2}=1; \\
\E_{5,2}^{(II)^+}   & \hspace{-0.5cm}L \sim 2E_1 +2E_2,      &  \hspace{0.2cm}E_1 \cdot E_2=1; \\
\E_{5,2}^{(II)^-}   & \hspace{0.5cm} L \sim 2E_1 +2E_2+K_S,   & \hspace{0.2cm} E_1 \cdot E_2=1
\end{eqnarray*}
(where '$\sim$' denotes linear equivalence). 
The components can also be distinguished by studying the projective models of its general members, which is classical, cf. \cite[Prop. 4.1.2, Prop. 4.5.1, Thm. 4.6.3, Prop. 4.7.1, Thm. 4.7.1]{cd}. 
These cases also furnish a nice sample of Theorem \ref{mainthm1}: under $\rho_5: \E_5 \to \widehat{\E}_5$, the two components $\E_{5,2}^{(II)^+}$ and $\E_{5,2}^{(II)^-}$ are identified, whereas $\E_{5,2}^{(I)}$ is mapped two-to-one onto one  irreducible component of $\widehat{\E}_5$.

To explain our results and our answer to question (2),
let $L$ be an effective line bundle on an Enriques surface satisfying $L^2>0$.
Set
\begin{equation} \label{eq:defeps}
\varepsilon_L= \begin{cases} 0, & \mbox{if $L+K_S$ is not $2$-divisible in $\Pic S$,} \\
1, & \mbox{if $L+K_S$ is $2$-divisible in $\Pic S$.} 
\end{cases}
\end{equation}
We will prove (cf. Theorem \ref{thm:uniquefund}) that there exist {\it unique} nonnegative integers $a_i$, depending on $L$, satisfying
\[ a_1\geq \cdots \geq a_7 \; \; \mbox{and} \; \;
a_9+a_{10} \geq a_0 \geq a_9 \geq a_{10} \]
such that $L$ can be written as\footnote{The reason for choosing to write \eqref{eq:scrivoL0} without the term ``$a_8E_8$'' is because then one automatically has $E_1 \cdot L \leq \cdots \leq E_{10} \cdot L$.}
\begin{equation} \label{eq:scrivoL0}
     L \sim a_1E_1+\cdots +a_7E_7+a_9E_9+a_{10}E_{10}+a_0E_{9,10}+\varepsilon_L K_S,
   \end{equation}
   for an isotropic $10$-sequence $\{E_1,\ldots,E_{10}\}$ of effective divisors (cf. Definition \ref{def:rseq}) and an effective isotropic divisor $E_{9,10} \sim \frac{1}{3}\left(E_1+\cdots+E_{10}\right)-E_9-E_{10}$
(cf. Lemma \ref{lemma:ceraprima}).   
 We call \eqref{eq:scrivoL0} a {\it fundamental presentation of $L$} and the coefficients $a_i=a_i(L)$ and $\varepsilon_L$ the {\it fundamental coefficients of $L$} (cf. Definitions \ref{def:fund} and \ref{def:fund2}). We will prove that {\it the irreducible components of $\E_{g}$ are precisely the loci parametrizing pairs of genus $g$ with the same fundamental coefficients} (cf. 
   Theorem \ref{mainthm2'}).

   As an alternative  description of the irreducible components of $\E_g$ we will introduce a new function on the Enriques lattice $\mathfrak{N}$ that generalizes the $\phi$-function defined in \eqref{eq:defphi}.
On the set of ordered $10$-tuples of integers one has an order relation  by setting $(a_1,\ldots,a_{10}) <(b_1,\ldots,b_{10})$ if either
$\sum_{i=1}^{10} a_i< \sum_{i=1}^{10} b_i$ or
$\sum_{i=1}^{10} a_i =\sum_{i=1}^{10} b_i$ and there is an $n \in \{1,\ldots,9\}$ such that $a_i = b_i$ for all $i \in \{1,\ldots,n-1\}$ and $a_n< b_n$.

\begin{definition} \label{def:phivector}
  Let $L$ be an effective line bundle on an Enriques surface $S$ such that $L^2 >0$. The {\em $\phi$-vector associated to $L$}, denoted by $\underline{\phi}(L)=(\phi_1(L),\ldots,\phi_{10}(L)) \in \ZZ_+^{10}$, is the minimal value of all $(E_1 \cdot L,\ldots,E_{10} \cdot L)$ under the above mentioned order relation, where $(E_1,\ldots,E_{10})$ runs over all isotropic $10$-sequences satisfying $0< E_1 \cdot L \leq \cdots \leq E_{10} \cdot L$. 

  We say that an isotropic $10$-sequence $\{E_1,\ldots,E_{10}\}$ {\em computes}
  $\underline{\phi}(L)$ if $E_i \cdot L=\phi_i(L)$ for all $i \in \{1,\ldots,10\}$. 
  \end{definition}

Thus, the $\phi$-vector function measures the ``lowest'' intersection numbers of line bundles with respect to entire isotropic $10$-sequences, generalizing Cossec's $\phi$-function, since, as proved in the next result, $\phi_1(L)=\phi(L)$. We will prove the following properties:

\begin{thm} \label{thm:phivector1}
Let $\underline{\phi}(L)=(\phi_1,\ldots,\phi_{10})$ be the
$\phi$-vector associated to an effective line bundle $L$ with $L^2>0$ on an Enriques surface $S$. Then
\begin{itemize}
\item[(a)] $0< \phi_1 \leq \cdots \leq \phi_{10}$;
\item[(b)] $\sum_{i=1}^{10}\phi_i$ is divisible by $3$;
\item[(c)] $\phi_1+\cdots+\phi_7 \geq  2\left(\phi_8+\phi_9+\phi_{10}\right)$;
\item[(d)] $L^2=\frac{1}{9}\left(\sum_{i=1}^{10}\phi_i\right)^2-\sum_{i=1}^{10}\phi_i^2$;
\item[(e)] $\phi_1,\ldots,\phi_8$ are the eight lowest intersection numbers with $L$ achieved by numerically distinct effective isotropic divisors on $S$; in particular $\phi_1=\phi(L)$;
 \item[(f)] $L$ is numerically $2$-divisible if and only if $\phi_i$ is even for all $i \in \{1,\ldots,10\}$;   
\item[(g)] the isotropic $10$-sequences computing $\underline{\phi}(L)$ are, up to numerical equivalence, precisely the ones appearing in fundamental presentations of $L$.
\end{itemize}

Conversely, for any Enriques surface $S$ and for any $10$-tuple of  integers $(\phi_1,\ldots, \phi_{10})$ satisfying (a)-(c), there is an $[L] \in \Num S$ such that $L^2 >0$ and $\underline{\phi}(L)=(\phi_1,\ldots,\phi_{10})$.
\end{thm}

In particular, we remark that the set of values $(g(L),\underline{\phi}(L))$ occurring as genus and $\phi$-vector of polarized Enriques surfaces $(S,L)$ are completely determined, and this a posteriori determines all possible values of pairs $(g(L),\phi(L))$ by property (e).

Our anwer to question (2) can now be summarized as:

\begin{thm} \label{mainthm3}
  The irreducible components of $\E_{g}$ are in one-to-one correspondence with the set of $11$-tuples of integers $(\phi_1,\ldots, \phi_{10},\varepsilon)$ satisfying
  \begin{itemize}
  \item[(i)] $0<\phi_1 \leq \cdots \leq \phi_{10}$,
  \item[(ii)] $\sum_{i=1}^{10}\phi_i$ is divisible by $3$,
  \item[(iii)]  $\phi_1+\cdots+\phi_7 \geq 2\left(\phi_8+\phi_9+\phi_{10}\right)$,
   \item[(iv)] $\varepsilon \in \{0,1\}$, with $\varepsilon=0$ occurring if at least one $\phi_i$ is odd,
  \item[(v)] $2g-2=\frac{1}{9}\left(\sum_{i=1}^{10}\phi_i\right)^2-\sum_{i=1}^{10}\phi_i^2$.
  \end{itemize}
Precisely, the irreducible component of $\E_{g}$ corresponding to a specific $(\phi_1,\ldots, \phi_{10},\varepsilon)$ parametrizes all pairs $(S,L)$ with $\underline{\phi}(L)= (\phi_1,\ldots, \phi_{10})$ and $\varepsilon_L=\varepsilon$, which are precisely the pairs with the following fundamental presentation, setting $s:=
\frac{1}{3}\sum_{j=1}^{10}\phi_j$:
\[ L \sim \sum_{i=1}^7(\phi_8-\phi_i)E_i+
(s-2\phi_8-\phi_9)E_9+(s-2\phi_8-\phi_{10})E_{10}
+(s-3\phi_8)E_{9,10}+\varepsilon K_S,  \]
for an isotropic $10$-sequence $\{E_1,\ldots,E_{10}\}$, with $E_{9,10} \sim \frac{1}{3}(E_1+\cdots + E_{10})-E_9-E_{10}$.
\end{thm}

We propose the following notation for the irreducible components of $\E_g$:
\begin{itemize}
\item $\E_{g;\phi_1,\ldots,\phi_{10}}$   corresponds to $(\phi_1,\ldots,\phi_{10})$ and $\epsilon=0$, if at least one $\phi_i$ is odd.
\item $\E_{g;\phi_1,\ldots,\phi_{10}}^+$ corresponds to $(\phi_1,\ldots,\phi_{10})$ and $\epsilon=0$, if all $\phi_i$ are even.
 \item $\E_{g;\phi_1,\ldots,\phi_{10}}^-$ corresponds to $(\phi_1,\ldots,\phi_{10})$ and $\epsilon=1$, if all $\phi_i$ are even.
  \end{itemize}
  Thus, by property (f) in Theorem \ref{thm:phivector1}, the components $\E_{g;\phi_1,\ldots,\phi_{10}}$ parametrize pairs $(S,L)$ with $L$ not numerically $2$-divisible, $\E_{g;\phi_1,\ldots,\phi_{10}}^+$ parametrize pairs $(S,L)$ with $L$ $2$-divisible in $\Pic S$, and $\E_{g;\phi_1,\ldots,\phi_{10}}^-$ parametrize pairs $(S,L)$ with $L+K_S$ $2$-divisible in $\Pic S$. In particular, by Theorem \ref{mainthm1} we obtain that
  {\it the map $\rho_g:\E_g \to \widehat{\E}_g$ identifies $\E_{g;\phi_1,\ldots,\phi_{10}}^+$ with $\E_{g;\phi_1,\ldots,\phi_{10}}^-$ and is two-to-one on $\E_{g;\phi_1,\ldots,\phi_{10}}$.}  We propose to use the notation $\widehat{\E}_{g;\phi_1,\ldots,\phi_{10}}$ for the images by $\rho_g$ of $\E_{g;\phi_1,\ldots,\phi_{10}}^{\pm}$ and $\E_{g;\phi_1,\ldots,\phi_{10}}$. As an application of our results we obtain a specific component dominating all others:

\begin{proposition} \label{prop:dom}
   For any $g \geq 2$ and any irreducible component
$\widehat{\E}_{g;\phi_1,\ldots,\phi_{10}}$ of $\widehat{\E}_g$ there is a surjective morphism from the  irreducible component $\widehat{\E}_{621;30,31,32,33,34,35,36,37,38,39}$   to $\widehat{\E}_{g;\phi_1,\ldots,\phi_{10}}$.
  \end{proposition}
 
It would be interesting to know whether this component is the same as the dominating component found by Gritsenko and Hulek \cite{GrHu}, cf. Question \ref{Q:same}.

The paper is organized as follows. In \S \ref{sec:iso} we recall and improve some results from \cite{cdgk} on effective decompositions of line  bundles on Enriques surfaces into isotropic divisors. In \S \ref{sec:degenr} we
study reducible surfaces that are a transversal union of a rational surface and
a surface birational to the second symmetric product of an elliptic curve, considered first in \cite{cdgk-deg}. Our main result is Theorem \ref{thm:hilbert}, which says that projective models in $\PP^{g-1}$ of those reducible surfaces by  line bundles of degree $2g-2$ (as described in Proposition \ref{prop:immersione}) are smoothable to Enriques surfaces of degree $2g-2$; more precisely, they represent smooth points in the Hilbert scheme of such surfaces. This result is the Enriques version of \cite[Thm. 1]{clm} for $K3$ surfaces  and we believe that it is of independent interest and that it will have further applications; indeed, although degenerations of Enriques surfaces have been widely studied (cf., e.g., \cite{K,PP,M,Shah}), a concrete result such as Theorem \ref{thm:hilbert} has not been available yet, cf. Remark \ref{rem:border}. A second key result is Proposition \ref{prop:main} stating that under certain conditions (which will turn out to be equivalent to numerical non-$2$-divisibility) the projective models by both a line bundle and its adjoint lie in the same irreducible component of the Hilbert scheme. 

In \S \ref{sec:proof} we prove  Theorem \ref{mainthm1} by degeneration, using the results from \S \ref{sec:degenr}. Finally, in \S \ref{sec:phi} we introduce the notions of fundamental presentation and $\phi$-vector mentioned above and prove Theorems \ref{thm:phivector1} and \ref{mainthm3}, as well as Proposition \ref{prop:dom}.

\vspace{0.3cm} 
\noindent
{\it Acknowledgements.} I thank Klaus Hulek for interesting correspondence about \cite{GrHu} and useful comments on a preliminary draft of this paper, Christian Liedtke, Frank Gounelas and Marian Aprodu for asking inspiring questions during a talk on mine on \cite{cdgk} at TU M{\"u}nchen, and Ciro Ciliberto, Thomas Dedieu, Concettina Galati, Michael Hoff, Yeongrak Kim, Frank-Olaf Schreyer and Alessandro Verra 
 for useful conversations on the topic.  I also thank a referee for a very careful reading and several useful comments and corrections, in particular for discovering a blunder in the proof of an incorrect first version of Proposition \ref{prop:dom}.  I acknowledge support from the Trond Mohn Foundation (project ``Pure Mathematics in Norway'') and grant 261756 of the Research Council of Norway.

\section{Isotropic $10$-sequences and simple isotropic decompositions} \label{sec:iso}

Let us explain some notions from \cite{cdgk}.  Any effective line bundle $L$
with $L^ 2\geqslant 0$ on an Enriques surface
may be written as  (cf. \cite[Cor. 4.6]{cdgk}) 
\begin {equation}\label{eq:ssid}  L \sim a_1E_1+\cdots+a_nE_{n}+\varepsilon K_S,
\end{equation}
such that  all $E_i$ are effective, non--zero, \emph{isotropic} (i.e., $E_i^2=0$) and \emph{primitive} (i.e., indivisible in $\Num S$), all $a_i$ are positive integers, $\varepsilon \in \{0,1\}$, $ n \leqslant 10$ and  
\begin{equation}
\begin{cases}
 \mbox{either  $n \neq 9$,  $E_i \cdot E_j=1$ for all $i \neq j$,} \\
\label{eq:int2-int} 
\mbox{or  $n \neq 10$,  $E_1 \cdot E_2=2$ and $E_i \cdot E_j=1$ for all other indices
  $i \neq j$,} \\
 \mbox{or  $E_1 \cdot E_2=E_1 \cdot E_3=2$ and $E_i \cdot E_j=1$ for all
  other indices $i \neq j$,} 
\end{cases}
\end{equation}
up to reordering indices.
We call this a {\it simple isotropic decomposition},  cf. \cite[Def. 4.1]{cdgk}.

We say that two polarized  Enriques surfaces $(S,L)$ and $(S',L')$ in $\E_g$  {\it admit the same  simple decomposition type} (cf. \cite[Def. 4.13]{cdgk}) if one 
 has simple isotropic decompositions 
\begin{equation}\label{eq:sdt}  L \sim a_1 E_1+\cdots +a_nE_n+\varepsilon K_S
\; \; \mbox{and} \; \; L' \sim a_1 E'_1+\cdots +a_nE'_n+\varepsilon K_{S'}, \; \;  \mbox{with} \; \; \varepsilon   \in \{0,1\} \end{equation}
such that $E_i \cdot E_j=E'_i \cdot E'_j$ for all $i \neq j$.
Similarly, we say that two numerically polarized Enriques surfaces $(S,[L])$ and $(S, [L'])$ in $\widehat{\E}_g$   admit the same  simple decomposition type if \eqref{eq:sdt} holds modulo $K_S$ and $K_{S'}$.

We note that $\varepsilon=1$ is only needed in \eqref {eq:sdt} when
all $a_i$s are even, otherwise one may substitute any $E_i$ having odd
coefficient with $E_i+K_S$.  Also note that a given line
bundle may admit decompositions  of different types, cf. \cite[Rmk. 4.14]{cdgk},  but nevertheless the property of admitting the same decomposition type is an equivalence relation on $\E_g$ and $\widehat{\E}_g$, cf. \cite[Prop. 4.15]{cdgk}.

 We recall the following from \cite[p.~122]{cd}:

\begin{definition} \label{def:rseq}
  An {\em isotropic $10$-sequence} on an Enriques surface $S$  is a sequence of isotropic effective divisors $\{E_1, \ldots, E_{10}\}$  such that $E_i \cdot E_j=1$ for $i \neq j$.
\end{definition}

It is  well-known that any Enriques surface contains such sequences. Note that we, contrary to \cite{cd}, require the divisors to be {\it effective}, which can always be arranged by changing signs. 
 We will  also   make use of the following result, cf. \cite[Lemma 3.4(a)]{cdgk}, \cite[Lemma 1.6.2(i)]{cos2} or \cite[Cor. 2.5.5]{cd}:

\begin{lemma} \label{lemma:ceraprima}
   Let $\{E_1,\ldots,E_{10}\}$ be an isotropic $10$-sequence. Then there exists a divisor $D$ on $S$ such that $D^2=10$, $\phi(D)=3$ and
$3D \sim E_1+\cdots+E_{10}$. Furthermore, for any $i \neq j$, we have
\begin{equation} \label{eq:10-3}
 D \sim E_i+E_j+E_{i,j}, \; \; \mbox{with $E_{i,j}$ effective isotropic,} \; \; E_i \cdot E_{i,j}=E_j \cdot E_{i,j}=2,
\end{equation} 
and 
$E_k \cdot E_{i,j}=1 \; \; \mbox{for} \; \; k \neq i,j$. 
 Moreover,
$E_{i,j} \cdot E_{k,l}= \begin{cases} 1, \; \mbox{if} \; \{i,j\} \cap \{k,l\} \neq \emptyset, \\
2, \; \mbox{if} \; \{i,j\} \cap \{k,l\} = \emptyset. \end{cases}$
\end{lemma}

In particular, for $i,j,k$ distinct, we have
$E_i+E_j+E_{i,j} \sim E_i+E_k+E_{i,k}$, so that
\begin{equation}
  \label{eq:ijk}
  E_j+E_{i,j} \sim E_k+E_{i,k}.
\end{equation}

The next result yields a ``canonical'' way of writing simple isotropic decompositions:

\begin{proposition} \label{prop:sid}
   Let $L$ be any effective line bundle on an Enriques surface $S$ such that $L^2>0$. Then there is an isotropic $10$-sequence $\{E_1,\ldots,E_{10}\}$ (depending on $L$) such that there is a simple isotropic decomposition
   \begin{equation} \label{eq:scrivoL}
     L \sim a_1E_1+\cdots +a_7E_7+a_9E_9+a_{10}E_{10}+a_0E_{9,10}+\varepsilon_L K_S,
   \end{equation}
   where $E_{9,10} \sim \frac{1}{3}\left(E_1+\cdots+E_{10}\right)-E_9-E_{10}$ and
$a_0,a_1,\ldots,a_{10}$ are nonnegative integers satisfying
\begin{eqnarray}
  \label{eq:condcoff}
  & a_1\geq \cdots \geq a_7, & \; \; \mbox{and} \\
  \label{eq:condcoff'}
  & a_9+a_{10} \geq a_0 \geq a_9 \geq a_{10}. &
\end{eqnarray}
\end{proposition}

\begin{proof}
  By \cite[Cor. 4.7]{cdgk} combined with  \cite[Rem. 4.11]{cdgk}, after renaming indices, there is an isotropic $10$-sequence $\{E_1,\ldots,E_{10}\}$ and nonnegative integers $a_0,a_1,\ldots,a_{10}$ such that
  \[
     L \sim a_1E_1+\cdots +a_{10}E_{10}+a_0E_{9,10}+\varepsilon_L K_S
   \]
   with $a_0=0$ or $a_8=0$. We have left to prove that we can make sure the coefficients satisfy  \eqref{eq:condcoff} and \eqref{eq:condcoff'}.

  Assume  $a_0=0$. By renaming indices we may assume $a_1\geq \cdots \geq a_{10}$, so that \eqref{eq:condcoff} is satisfied. If $a_8=0$, then $a_9=a_{10}=0$ and \eqref{eq:condcoff'} is satisfied. If $a_8>0$, then, using \eqref{eq:10-3}:
  \begin{eqnarray*}
    L & \sim & a_1E_1+\cdots +a_{10}E_{10}+\varepsilon_L K_S \\
      & \sim & a_8(E_1+\cdots+E_{10})+(a_1-a_8)E_1+\cdots+(a_{10}-a_8)E_{10}+\varepsilon_L K_S \\
      & \sim &  3a_8(E_9+E_{10}+E_{9,10})+(a_1-a_8)E_1+\cdots+(a_{10}-a_8)E_{10}+\varepsilon_L K_S \\
      & \sim & \sum_{i=1}^7(a_i-a_8)E_i+(2a_8+a_9)E_9+(2a_8+a_{10})E_{10}+3a_8E_{9,10}+\varepsilon_L K_S.
  \end{eqnarray*}
  Setting $a_i':=a_i-a_8$ for $i \in \{1,\ldots,7\}$, $a_i':=2a_8+a_i$ for $i \in \{9,10\}$ and $a_0':=3a_8$, we see that $a'_1\geq \cdots \geq a'_7$ and 
  \[ a'_9+a'_{10}=4a_8+a_9+a_{10} \geq 3a_8=a'_0 \geq 2a_8+a_9=a'_9 \geq 2a_8+a_{10}=a'_{10},\]
  so the coefficients $a'_i$ satisfy \eqref{eq:condcoff} and \eqref{eq:condcoff'}.

  Assume henceforth that $a_0>0$, so that $a_8=0$. By renaming indices we may assume $a_1\geq \cdots \geq a_{7}$, so that \eqref{eq:condcoff} is satisfied, and $a_9 \geq a_{10}$. We see that \eqref{eq:condcoff'} is satisfied unless $a_0 < a_9$ or $a_9+a_{10} <a_0$. We treat these two cases separately.

  {\bf Case $a_0 < a_9$.} We set $b:=\min\{a_9-a_0,a_7\}$.
Recalling \eqref{eq:10-3} and \eqref{eq:ijk}, we have
\begin{eqnarray*}
   L & \sim & a_1E_1+\cdots + a_7E_7+a_9E_9+a_{10}E_{10}+a_0E_{9,10}+\varepsilon_L K_S \\
         & \sim & b(E_1+\cdots+E_{10})+\sum_{i=1}^7(a_i-b)E_i-bE_8+a_0(E_9+E_{9,10}) \\
  & & \hspace{2cm} +(a_9-a_0-b)E_9+(a_{10}-b)E_{10} +\varepsilon_L K_S \\
           & \sim & 3b(E_8+E_{10}+E_{8,10})+\sum_{i=1}^7(a_i-b)E_i-bE_8+ a_0(E_8+E_{8,10})\\
  & & \hspace{2cm} +(a_9-a_0-b)E_9+(a_{10}-b)E_{10} +\varepsilon_L K_S \\
  & \sim & \sum_{i=1}^7(a_i-b)E_i+(a_9-a_0-b)E_9\\
  & & \hspace{2cm} +(a_0+2b)E_8+(a_{10}+2b)E_{10}+(a_0+3b)E_{8,10}+\varepsilon_L K_S.
  \end{eqnarray*}

  By definition of $b$, we see that at least one among $E_7$ and $E_9$ appears with  coefficient  $0$. Hence, in the expression $\sum_{i=1}^7(a_i-b)E_i+(a_9-a_0-b)E_9$ there are at most $7$ nonzero terms, and we may rearrange them so that the coefficients appear in decreasing order, that is, so that \eqref{eq:condcoff} is satisfied. Moreover, we see that $(a_0+2b)+(a_{10}+2b)=a_0+a_{10}+4b \geq a_0+3b$. If $a_0 \geq a_{10}$, we see that also $a_0+3b \geq a_0+2b \geq a_{10}+2b$, whence \eqref{eq:condcoff'} is satisfied. If instead $a_0 < a_{10}$, we set $E'_9:=E_{10}$, $a'_9:=a_{10}+2b$, $E'_{10}:=E_8$, $a'_{10}:=a_0+2b$, $E'_{9,10}:=E_{8,10}$ and $a'_0:=a_0+3b$; then we rewrite
  \[ (a_0+2b)E_8+(a_{10}+2b)E_{10}+(a_0+3b)E_{8,10}=a'_9E'_9+a'_{10}E'_{10}+a'_0E'_{9,10},
    \]
with $a'_9+a'_{10} \geq a'_0$, $a'_9 >a'_{10}$ and $a'_0 \geq a'_{10}$. If $a'_0 \geq a'_9$ (which happens if and only if $a_0+b \geq a_{10}$), we are done. If $a'_0 <a'_9$, we repeat the process from the start, which this time will give the desired decomposition, since $a'_0 \geq a'_{10}$.

{\bf Case $a_9+a_{10} <a_0$.} Recalling \eqref{eq:ijk}, we have
  \begin{eqnarray*}
    L & \sim & a_1E_1+\cdots + a_7E_7+a_9E_9+a_{10}E_{10}+a_0E_{9,10}+\varepsilon_L K_S \\
         & \sim & \sum_{i=1}^7a_iE_i+a_9(E_9+E_{9,10})+
                              a_{10}(E_{10}+E_{9,10})+(a_0-a_9-a_{10})E_{9,10} +\varepsilon_L K_S \\
         & \sim & \sum_{i=1}^7a_iE_i+a_9(E_8+E_{8,10})+
                              a_{10}(E_8+E_{8,9}) +(a_0-a_9-a_{10})E_{9,10} + \varepsilon_L K_S \\
     & \sim & \sum_{i=1}^7a_iE_i + (a_9+a_{10})E_8+a_9E_{8,10}+a_{10}E_{8,9}+(a_0-a_9-a_{10})E_{9,10}+  \varepsilon_L K_S.
     \end{eqnarray*}
     We note that $\{E_1,\ldots,E_7,E_{8,10},E_{8,9},E_{9,10}\}$ is an isotropic $10$-sequence, and $E_8 \sim \frac{1}{3}\left(E_1+\cdots+E_7+E_{8,10}+E_{8,9}+E_{9,10}\right)-E_{8,10}-E_{8,9}$ (cf. Lemma \ref{lemma:ceraprima}). Thus, setting  $E'_8:=E_{9,10}$,
     $E'_9:=E_{8,10}$,  $E'_{10}:=E_{8,9}$ and  $E'_{9,10}:=E_8$, and $b:=\min\{a_7,a_0-a_9-a_{10}\}$,
 we may rewrite as
 \begin{eqnarray*}
L     & \sim &  \sum_{i=1}^7a_iE_i+(a_0-a_9-a_{10})E'_8+a_9E'_9+a_{10}E'_{10}+(a_9+a_{10})E'_{9,10} + \varepsilon_L K_S \\
      & \sim & b(E_1+\cdots+E_7+E'_8+E'_9+E'_{10})+\sum_{i=1}^7(a_i-b)E_i +(a_0-a_9-a_{10}-b)E'_8\\
           & & \hspace{2cm} +(a_9-b)E'_9+(a_{10}-b)E'_{10}+(a_9+a_{10})E'_{9,10} + \varepsilon_L K_S \\
 & \sim & 3b(E'_9+E'_{10}+E'_{9,10})+ \sum_{i=1}^7(a_i-b)E_i +(a_0-a_9-a_{10}-b)E'_8\\
           & & \hspace{2cm} +(a_9-b)E'_9+(a_{10}-b)E'_{10}+(a_9+a_{10})E'_{9,10} + \varepsilon_L K_S \\
           & \sim & \sum_{i=1}^7(a_i-b)E_i +(a_0-a_9-a_{10}-b)E'_8  \\
           & & \hspace{2cm}+ (a_9+2b)E'_9+(a_{10}+2b)E'_{10}+ (3b+a_9+a_{10})E'_{9,10} +\varepsilon_L K_S.                  
 \end{eqnarray*} 
 By definition of $b$, we see that at least one of $E_7$ and $E'_8$ appears with coefficient $0$. Hence, in the expression $\sum_{i=1}^7(a_i-b)E_i+(a_0-a_9-a_{10}-b)E'_8$ there are at most $7$ nonzero terms, and we may rearrange them so that the coefficients appear in decreasing order, that is, so that \eqref{eq:condcoff} is satisfied. We also see that the coefficients in front of $E'_9$, $E'_{10}$ and  $E'_{9,10}$ satisfy the conditions \eqref{eq:condcoff'}.

Finally, the fact that \eqref{eq:scrivoL} is a simple isotropic decomposition
is easy to check.
\end{proof}

\begin{remark} \label{rem:eps}
  Recalling \eqref{eq:defeps}, we have by \cite[Lemma 4.8]{cdgk}
  that
  \[
    \varepsilon_L= \begin{cases} 0, & \mbox{if some $a_i$ is odd} \\
0 \; \mbox{or} \; 1, & \mbox{if all $a_i$ are even,} 
\end{cases}
\]
with the $a_i$s as in Proposition \ref{prop:sid}.
\end{remark}

\section{Flat limits of Enriques surfaces} \label{sec:degenr}

Let $E$ be a smooth elliptic curve. Denote by $\+$ (and $\ominus$) the group operation on $E$  and by $e_0$ the neutral element. Let $R:=\Sym^2(E)$ and
$\pi: R \to E$ be the (Albanese) projection map sending $x+y$ to $x\+ y$.
We denote the fiber of $\pi$ over a point $e \in E$ by 
\[ \f_e:=\pi^{-1}(e)=\{ x+y \in \Sym^2(E) \; | \; x\+ y=e \; \mbox{(equivalently,} \; 
x + y \sim e+e_0)\},\]
which is the $\PP^1$ defined by the linear system $|e+e_0|$. We denote the algebraic equivalence class of the fibers by $\f$. 

For each $e \in E$, we define the curve $\s_e$ (called $D_e$ in \cite{CaCi})
as the image of the section $E \to R$ mapping $x$ to $e+ (x \ominus e)$.
We let $\s$ denote the algebraic equivalence class of these sections, which are the ones with minimal self-intersection, namely $1$, cf. \cite{CaCi}.  We note for later use that
 for $x \neq y$ we have
\begin{equation} \label{eq:duesez} 
\s_x \cap \s_y=\{ x+y\}.
\end{equation}
We also note that we have
\begin{equation} \label{eq:can} 
K_R \sim -2\s_{e_0}+\f_{e_0}.
\end{equation}

For any of the three nonzero $2$-torsion points $\eta$ of $E$ the map $E \to R$ defined by mapping $e$ to $e + (e \+ \eta)$ realizes $E$ as an unramified double cover of its image curve
\[ T_{\eta}:= \{ e+ (e\+\eta) \; | \; e \in E\}. \]
We have
\begin{equation} \label{eq:T} 
T_{\eta} \sim -K_R+\f_{\eta}-f_{e_0} \sim 2\s_{e_0}-2\f_{e_0}+f_{\eta},
\end{equation}
by \cite[(2.10)]{CaCi}. In particular,
\begin{equation} \label{eq:T2} 
  T_{\eta} \not \sim -K_R \; \; \mbox{and} \; \; 2T_{\eta} \sim -2K_R.
\end{equation}

We henceforth fix $\eta$ and set $T:=T_{\eta}$.
For later use we gather a  few  lemmas here:

\begin{lemma} \label{lemma:g1}
We have $h^i(\T_R(-T))=0$ for all $i$.
\end{lemma}

\begin{proof}
We first note that \eqref{eq:T}, Serre duality and Riemann-Roch imply that
\begin{equation}
  \label{eq:coh-K-T}
  h^i(\O_R(-K_R-T))=h^i(\O_R(-T))=0 \; \; \mbox{for all} \; \; i.
\end{equation}
Then the lemma follows from the sequence
\[
 0 \longrightarrow \O_R(-K_R-T) \longrightarrow \T_R(-T) \longrightarrow \O_R(-T) \longrightarrow 0,
\]
which is the dual of the sequence of relative differentials of $\pi$ tensored by $\O_R(-T)$.
\end{proof}

\begin{lemma} \label{lemma:g2}
We have 
\begin{equation}
  \label{eq:relsf}
  \s_x + \f_y \sim \s_y+\f_x \; \; \mbox{for all} \; \; x,y \in E.
\end{equation}
In particular,
\begin{equation}
  \label{eq:relsf2}
  \s_x + \f_{\eta} \sim \s_{x\+\eta}+\f_{e_0} \; \; \mbox{for all} \; \; x \in E.
\end{equation}
\end{lemma}

\begin{proof}
Restricting to $\s_{e_0}$ and using the isomorphism $\pi|_{\s_{e_0}}:\s_{e_0} \to E$, we have
\[ (\s_x-\s_y)|_{\s_{e_0}} \sim \pi|_{\s_{e_0}}^*(x-y) \sim (\f_x-\f_y)|_{\s_{e_0}},\]
and \eqref{eq:relsf} follows from the special case of \cite[Prop. (2.11)]{CaCi}
stating that two line bundles on $R$ with the same restriction to a section are isomorphic. Setting $y:=x\+\eta$ in \eqref{eq:relsf} and using the group law ($x \+ y \sim x+y-e_0$) on $E$, we obtain \eqref{eq:relsf2}.
  \end{proof}

\begin{lemma} \label{lemma:restaT}
  If $\B \in \Pic^0 E$ such that $\pi^*\B|_{T}$ is trivial, then $\B \cong \O_E$ or $\B \cong \O_E(\eta-e_0)$.
\end{lemma}

\begin{proof}
  Any $\B \in \Pic^0 E$ can be written as $\B \cong \O_E(x-e_0)$ for some $x \in E$. Assume that $x \neq e_0$. Since $\pi^*\B|_{T} \cong \O_T(\f_x-\f_{e_0})$ is trivial, we have, using \eqref{eq:T}, a short exact sequence 
\[
 0 \longrightarrow \O_R(\f_x-\f_{e_0}-T)\cong \O_R(K_R+\f_x-\f_{\eta})  \longrightarrow \O_R(\f_x-\f_{e_0}) \longrightarrow \O_{T} \longrightarrow 0.
\]
As $h^i(\O_R(\f_x-\f_{e_0}))=0$ for $i=0,1,2$, we must have $h^2(\O_R(K_R+\f_x-\f_{\eta}))=h^1(\O_T)=1$, whence $x=\eta$ by Serre duality, finishing the proof.
\end{proof}

Embed $T$ as a cubic in $\PP^2$. Consider distinct points  $x_1,\ldots,x_9 \in T$ such that 
\begin{equation}
  \label{eq:cond}
  x_1+\cdots+x_9 \in |\N_{T/R} \* \N_{T/\PP^2}|.
\end{equation}
Let $\sigma_R:\widetilde{R} \to R$ be the blow up at $x_1,x_9$ and $\sigma_P:\widetilde{P} \to P:=\PP^2$ be the blow up at $x_2,\ldots,x_8$. We will always assume the points $x_i$ to be sufficiently general, so that
\begin{equation} \label{eq:dp}
  \widetilde{P} \; \; \mbox{is Del Pezzo (whence contains no $(-2)$-curves), and }
\end{equation}
\begin{equation} \label{eq:df}
\mbox{$x_1$ and $x_9$ lie on distinct fibers of $\pi:R \to E$.}
\end{equation}
We denote by $\ell$ on $\widetilde{P}$ the pullback of a general line on $\PP^2$ and by $\e_i$ the exceptional divisor over $x_i$, $ i \in \{2,\ldots,8\}$. By abuse of notation we denote by $\s$ and $\f$ the pullbacks of sections and fibers on $\widetilde{R}$ and by $\e_i$ the exceptional divisor over $x_i$, $ i \in \{1,9\}$. We still denote by $\pi$ the composed map
$\widetilde{R} \to R \to E$. 
By abuse of notation we denote by $T$ the strict transform of $T$ in both $\widetilde{R}$ and $\widetilde{P}$. We have (cf. \eqref{eq:T}-\eqref{eq:T2})
\begin{eqnarray} \label{eq:T3} 
  T \sim   2\s_{e_0}-2\f_{e_0}+\f_{\eta}-\e_1-\e_9 \not \sim -K_{\widetilde{R}}, \; \; \; 2T  \sim -2K_{\widetilde{R}} & \; \; \mbox{on} \; \; \widetilde{R}, \\
 \label{eq:TP} T \sim   3\ell-\e_2-\cdots-\e_8  \sim -K_{\widetilde{P}} & \; \; \mbox{on} \; \; \widetilde{P}.
\end{eqnarray}
Define
$X:=\widetilde{R} \cup_T \widetilde{P}$ as the surface obtained by  gluing
$\widetilde{R}$ and $\widetilde{P}$ along $T$. 
The {\it first cotangent sheaf} $T^1_{X}:=\Shext^1_{\O_X}(\Omega_{X},\O_{X})$ of $X$ (cf. \cite[Cor. 1.1.11]{ser} or \cite[\S 2]{fri}), 
satisfies
\begin{equation} \label{eq:ss}
 T^1_{X} \cong \N_{T/\widetilde{R}} \* \N_{T/\widetilde{P}} \cong \O_T.
\end{equation}
by \cite[Prop. 2.3]{fri}, because of \eqref{eq:cond}.  
Thus, $X$ is {\it semi-stable}, cf. \cite[Def. (1.13)]{fri} and \cite[(0.4)]{fri2}. We will denote by $\D$ the family of surfaces $X$ obtained in this way.
It is easy to see that $\D$ is irreducible of dimension $9$.

\begin{lemma} \label{lemma:restcoho}
  We have $h^0(\O_X)=1$, $h^1(\O_X)=h^2(\O_X)=0$. In particular, $\Pic X \cong H^2(X,\ZZ)$.
\end{lemma}

\begin{proof}
Consider the decomposition sequence
\[ \xymatrix{
      0 \ar[r] & \O_{\widetilde{R}}(-T)  \ar[r] & \O_X \ar[r] &  \O_{\widetilde{P}}  \ar[r] & 0. }\]
  Since $h^0(\O_{\widetilde{R}}(-T))=0$ and $h^2(\O_{\widetilde{R}}(-T))=h^0(\O_{\widetilde{R}}(K_{\widetilde{R}}+T))=0$
(because $T$ is not anticanonical on $\widetilde{R}$), 
also $h^1(\O_{\widetilde{R}}(-T))=0$ by Riemann-Roch.
We thus get $h^j(\O_X)=h^j(\O_{\widetilde{P}})$ for $j=0,1,2$, which has the stated values, as $\widetilde{P}$ is rational. The last statement follows from the cohomology of the exponential sequence.
\end{proof}

We recall that a Cartier divisor, or a line bundle, $\L \in \Pic X$, is a pair
$(L_{\widetilde{R}},L_{\widetilde{P}})$ such that $L_{\widetilde{R}} \in \Pic {\widetilde{R}}$, $L_{\widetilde{P}} \in \Pic {\widetilde{P}}$ and $L_{\widetilde{R}}|_T \cong
L_{\widetilde{P}}|_T$. We remark that since $T$ is numerically equivalent to the anticanonical divisor on both $\widetilde{R}$ and $\widetilde{P}$, we have
\begin{equation} \label{eq:even}
 \L^2=L_{\widetilde{R}}^2+L_{\widetilde{P}}^2=2p_a(L_{\widetilde{R}})-2+2p_a(L_{\widetilde{P}})-2+2d, \; \; d:=L_{\widetilde{R}}\cdot T=L_{\widetilde{P}}\cdot T,
\end{equation}  
so is even. As $\O_T(K_{\widetilde{R}}+T) \cong \omega_T \cong \O_T$, the canonical divisor $K_{X}$ is represented by 
\begin{equation}
  \label{eq:canrist}
  K_X=(K_{\widetilde{R}}+T,0)=(\f_{\eta}-\f_{e_0},0) \; \; \mbox{in} \; \; \Pic {\widetilde{R}} \x \Pic {\widetilde{P}}. 
\end{equation}
 In particular, by \eqref{eq:T3}-\eqref{eq:TP} we have
\begin{equation}
  \label{eq:doppioz}
  K_X \neq 0 \; \; \mbox{and} \; \; 2K_X=0.
\end{equation}
By \cite[(3.3)]{fri2} the surface $X$ also carries a Cartier divisor $\xi$ represented by the pair
\begin{equation} \label{eq:xi}
  \xi= (T,-T)\sim (2\s_{e_0}-2\f_{e_0}+\f_{\eta}-\e_1-\e_9,-3\ell+\e_2+\cdots+\e_8)
\; \; \mbox{in} \; \; \Pic {\widetilde{R}} \x
  \Pic {\widetilde{P}}.
\end{equation}

\begin{lemma} \label{lemma:onlytor}
  The Cartier divisor $K_X$ is the only nonzero torsion element of $\Pic X$.
\end{lemma}

\begin{proof}
  Assume that $\L$, represented by $(L_{\widetilde{R}},L_{\widetilde{P}})$ as above, is a torsion element of $\Pic X$. 
Since $\Pic \widetilde{P}$ is torsion-free, we must have $L_{\widetilde{P}}=\O_{\widetilde{P}}$. Moreover, since $L_{\widetilde{R}}$ is torsion in $\Pic \widetilde{R} \cong \sigma_R^* \Pic R \+ \ZZ[\e_1]\+\ZZ[\e_9]$, we must have $L_{\widetilde{R}} \cong \pi^* \B$ for a $\B \in \Pic^0(E)$
such that $\pi^*\B|_T \cong L_{\widetilde{P}}|_T \cong \O_T$.  By Lemma \ref{lemma:restaT} we have $\pi^*\B \cong \O_{\widetilde{R}}$ or $\pi^*\B \cong \O_{\widetilde R}(\f_{\eta}-f_{e_0})$. Thus, by \eqref{eq:canrist}, we have 
$\L \cong \O_X$ or $\L \cong \O_X(K_{X})$, as desired.
\end{proof}

We now find special effective primitive isotropic divisors on $X$ that will be used later.

For $j \in \{1,9\}$, the linear system $|\ell \* \I_{x_j}|$ on $P$ is a pencil inducing a 
$g^1_2$ on $T$, 
 which has, by Riemann-Hurwitz, two members that also belong to a fiber of $\pi_{|T}: T \to E$. In other words, there are two fibers $\f_{\alpha_j}$ and $\f_{\alpha'_j}$ of $\pi:\widetilde{R} \to E$ such that
\[
  (\f_{\alpha_j} \cup \e_j) \cap T  \in |\ell||_T \; \; \mbox{and} \; \;
 (\f_{\alpha'_j} \cup \e_j) \in |\ell||_T,\; \; j \in \{1,9\}.
\]
One easily verifies that $\alpha'_j=\alpha_j \+\eta$. In particular, there are two uniquely defined points $\alpha_j$ and $\alpha_j \+ \eta$ on $E$ such that the pairs
\[
  (\f_{\alpha_j}+\e_j,\ell) \; \; \mbox{and} \; \; (\f_{\alpha_j\+\eta}+\e_j,\ell),\; \; j \in \{1,9\},
\]
define Cartier divisors on $X$. It is easy to check, using the group law on $E$ and \eqref{eq:canrist}, that one is obtained from the other by tensoring with $K_{X}$.
We define
\[
  E_9:=(\f_{\alpha_9}+\e_9,\ell) \; \; \mbox{and} \; \; E_9+K_{X}=(\f_{\alpha_9\+\eta}+\e_9,\ell).
\]

Similarly, for each $i \in \{2,\ldots,8\}$ there are two uniquely defined points $\alpha_i$ and $\alpha_i \+ \eta$ on $E$ such that the pairs
\[
  E_i:=(\f_{\alpha_i},\ell-\e_i) \; \; \mbox{and} \; \; E_i+K_{X}=(\f_{\alpha_i\+\eta},\ell-\e_i), \; \; \mbox{for} \; \; i \in\{2,\ldots,8\},
\]
define Cartier divisors on ${X}$. 

Considering each point $x_i \in T$, for $i \in \{1,9\}$  as a point in $R=\Sym^2(E)$ we may write
$x_i=p_i+(p_i \+ \eta)$. There are two sections in $R$ passing through $x_i$, namely
$\s_{p_i}$ and $\s_{p_i \+ \eta}$, cf. \eqref{eq:duesez}. Thus, on $\widetilde{R}$ the pairs
\begin{equation}
  \label{eq:car2}
  (\s_{p_i}-\e_i,0) \; \; \mbox{and} \; \; (\s_{p_i\+ \eta}-\e_i,0), \; \; i \in\{1,9\}
\end{equation}
define Cartier divisors on ${X}$. Using \eqref{eq:relsf2} and \eqref{eq:canrist} one checks that one is obtained from the other by tensoring with $K_{X}$.
We define
\begin{equation} \label{eq:defE9}
  E_1:=(\s_{p_1}-\e_1,0) \; \; \mbox{and} \; \; E_1+K_{X}=(\s_{p_1\+ \eta}-\e_1,0).
\end{equation}

We have $\pi(x_i)=p_i \+ p_i \+ \eta$ and  $\f_{\pi(x_i)}$ is the unique fiber of $\pi: \widetilde{R} \to E$ passing through $x_i$. Its second intersection point with $T$ is 
$x'_i:= (p_i \+ \eta_1)+(p_i \+\eta_2)$, where $\eta_1$ and $\eta_2$ are the two  nonzero $2$-torsion points of $E$ in addition to $\eta$. The $g^1_2$ cut out on $T$ by the pencil of lines through $x'_i$ has, again by Riemann-Hurwitz as above, two elements that are fibers of  $\pi_{|T}: T \to E$,
say the fibers over $\beta_i \in E$ and $\beta_i \+ \eta \in E$. It follows that
for $i \in \{1,9\}$ there are two uniquely defined points $\beta_i$ and $\beta_i \+ \eta$ on $E$ such that the pairs
\begin{equation}
  \label{eq:car3}
  (\f_{\pi(x_i)}+ \f_{\beta_i}-\e_i, \ell) \; \; \mbox{and} \; \; 
(\f_{\pi(x_i)}+ \f_{\beta_i\+ \eta}-\e_i, \ell), \; \; i \in\{1,9\},
\end{equation}
define Cartier divisors on ${X}$. It is again easy to check that one is obtained from the other by tensoring with $K_{X}$. We define
\begin{equation} \label{eq:defE10}
  E_{10}:=(\f_{\pi(x_1)}+ \f_{\beta_1}-\e_1, \ell) \; \; \mbox{and} \; \; E_{10}+K_{X}=
(\f_{\pi(x_1)}+ \f_{\beta_1\+ \eta}-\e_1, \ell).
\end{equation}

Note that we have $E_i^2=0$ for all $i$ and $E_i \cdot E_j=1$ for all $i \neq j$. 

\begin{lemma} \label{lemma:div3}
We have $E_1+\cdots+E_{10}+\xi \sim 3(E_9+E_{10}+E_{9,10})$, with (cf. \eqref{eq:car2})
\begin{equation}
  \label{eq:110}
  E_{9,10} =(\s_{p_9}-\e_9,0) \; \; \mbox{and} \; \; E_{9,10}+K_{X}=(\s_{p_9\+\eta}-\e_9,0).
\end{equation}
\end{lemma}

 \begin{proof}
   By  \eqref{eq:relsf} we have $\s_{e_0} \sim \s_{p_9}+\f_{e_0}-\f_{p_9}$ and
   $\s_{p_1} \sim \s_{p_9}+\f_{p_1}-\f_{p_9}$. Hence one finds
   \begin{equation} \label{eq:A00}
     (E_1+\cdots+E_{10}+\xi) -3(E_9+E_{10}) \sim (3(\s_{p_9}-\e_9)+A,0) 
   \end{equation}
   with
   \[
     A:=\f_{p_1}+\f_{\eta}+\f_{\alpha_2}+\cdots+ \f_{\alpha_8}-3\f_{p_9}-2\f_{\alpha_9}-2\f_{\pi(x_1)}-2\f_{\beta_1} \equiv 0
   \]
 (where '$\equiv$' denotes numerical equivalence). Since $(\s_{p_9}-\e_9,0)$ is Cartier (cf. \eqref{eq:car2}) and all divisors on the left side of \eqref{eq:A00} are Cartier, we see that $(A,0)$ is Cartier as well. Since it is torsion in $\Pic X$, it equals $0$ or $K_X$ by Lemma \ref{lemma:onlytor}. Thus, \eqref{eq:A00} reads  
\[
     (E_1+\cdots+E_{10}+\xi) -3(E_9+E_{10}) \sim 3(\s_{p_9}-\e_9,0) \; \; \mbox{or} \; \;    3(\s_{p_9}-\e_9,0) +K_X. \]
The result follows possibly after interchanging $E_{9,10}$ and $E_{9,10}+K_X$, equivalently, $p_9$ and $p_9\+\eta$.
\end{proof}

Thus, we may similarly to \eqref{eq:10-3} define 
\begin{equation} \label{eq:defEij}
E_{i,j}:=\frac{1}{3}\left(E_1+\cdots+E_{10}+\xi\right)-E_i-E_j \; \; \mbox{for each} \; \; i \neq j.
\end{equation}
Hence \eqref{eq:ijk} holds on $X$. 
In particular, we remark for later that
\begin{equation}
  \label{eq:19}
   E_{1,9} \sim (\f_{\pi(x_9)}+ \f_{\beta_9}-\e_9,\ell) \; \; \mbox{and} \; \; E_{1,9}+K_{X} \sim
(\f_{\pi(x_9)}+ \f_{\beta_9\+ \eta}-\e_9, \ell)
\end{equation}
(cf. \eqref{eq:car3}, possibly after interchanging $\beta_9$ and $\beta_9 \+ \eta$).

\begin{proposition} \label{prop:immersione}
  Let $X=\widetilde{R} \cup_T \widetilde{P}$ be a member of $\D$ and
  \begin{equation} \label{dec:immersione}
    L \sim a_0E_{9,10}+a_1E_1+a_2E_2+\cdots+a_7E_7+a_9E_9+a_{10}E_{10},
    \end{equation}
with all $a_i \geq 0$ satisfying
\begin{eqnarray}
  \label{eq:condcoff2} & a_9+a_{10} \geq a_0 \geq \max\{a_9,a_{10}\},  & \\
  \label{eq:condcoffJ1}  &a_0+\min\{a_1,a_2\}>0,  & \\
  \label{eq:condcoffJ2}  & \min\{a_1,a_2\} \geq a_3 \geq \cdots \geq a_7,  & \\
 \label{eq:condcoffJ3}  & a_0+\min\{a_1,a_2\}+a_3+\cdots+a_7+a_9+a_{10} \geq 3. &
\end{eqnarray}
 Set $g:=\frac{1}{2}L^2+1$.  Then the complete linear system $|L|$ defines a morphism $\varphi_L:X \to \PP^{g-1}$ that is an isomorphism onto its image except for the contraction of $(-1)$-curves on $\widetilde{R}$ and $\widetilde{P}$ and it contracts at least one such curve, namely $\e_8$ on $\widetilde{P}$. Its image is
$\overline{X}:= \overline{R} \cup_{\overline{T}} \overline{P}$,
  where $\overline{R}$ and $\overline{P}$ are the images of $\widetilde{R}$ and  $\widetilde{P}$, respectively, and intersect transversally and only along $\overline{T}:=\varphi_L(T) \cong T$.

Furthermore,
\begin{itemize}
\item[(i)] $H^j(\overline{X},\O_{\overline{X}})=0$ for $j=1,2$;

\item[(ii)] $K_{\overline{X}}$ is Cartier and represented by $(K_{\overline{R}}+\overline{T},0)$, whence $K_{\overline{X}} \neq 0$ and $2K_{\overline{X}}=0$.
  \end{itemize}
  \end{proposition}

\begin{proof}
Set $L_{\widetilde{R}}:=L|_{\widetilde{R}}$ and $L_{\widetilde{P}}:=L|_{\widetilde{P}}$.  We  have
\begin{eqnarray*}
    & E_{9,10} \equiv (\s-\e_9,0), \; \; E_9\equiv(\f+\e_9,\ell), \; \; E_{10} \equiv (2\f- \e_1, \ell), & \\
   &E_1 \equiv (\s-\e_1,0), \; \;  E_i \equiv (\f,\ell-\e_i), i \in \{2,\ldots,7\}. &
  \end{eqnarray*}

\begin{claim} \label{cl:1}
  $L_{\widetilde{R}}$ is nef, $L_{\widetilde{R}}^2 \geq 5$ and $L_{\widetilde{R}} \cdot T \geq 5$. In particular, $L_{\widetilde{R}}+T$ is nef with $(L_{\widetilde{R}}+T)^2 \geq 13$.
\end{claim}

\begin{proof}[Proof of claim]
  We have an effective decomposition
\begin{equation} \label{eq:effI}
  L_{\widetilde{R}} \equiv a_0(\s-\e_9)+a_1(\s-\e_1)+a_{10}(\f-\e_1)+
  (a_2+\cdots+a_7+a_9+a_{10})\f+a_9\e_9.
    \end{equation}
The only negative components  are $\e_9$ and $\f-\e_1$. Since  
$\e_9 \cdot L_{\widetilde{R}}=a_0-a_{9} \geq 0$ and $(\f-\e_1)\cdot L_{\widetilde{R}}=a_0-a_{10} \geq 0$ (using \eqref{eq:condcoff2}), we see that $L_{\widetilde{R}}$ is nef. From \eqref{eq:effI}
we find
 \begin{equation} \label{eq:intsuR}
  L_{\widetilde{R}}^2  =  2(a_0+a_1)(a_2+\cdots+a_7+a_9+a_{10})+a_0(2a_1+a_9+a_{10}) +a_9(a_0-a_9) +a_{10}(a_0-a_{10}).
  \end{equation}
  One now readily checks that conditions \eqref{eq:condcoff2} and
  \eqref{eq:condcoffJ3} imply $L_{\widetilde{R}}^2 \geq 5$, as desired.

Finally, recalling \eqref{eq:T3}, we have $T^2=-2$ and $T \cdot L_{\widetilde{R}}=2(a_2+\cdots+a_7)+3(a_9+a_{10})$. Again \eqref{eq:condcoff2} and
  \eqref{eq:condcoffJ3} yield that $T \cdot L_{\widetilde{R}} \geq 5$. Since $T$ is irreducible with $T^2=-2$, it  follows that $L_{\widetilde{R}}+T$ is nef with $(L_{\widetilde{R}}+T)^2 \geq 13$.
\end{proof}

\begin{claim} \label{cl:2}
  $L_{\widetilde{R}}$ and  $L_{\widetilde{P}}$ are  globally generated and each defines a morphism that is an isomorphism except for the contraction of $(-1)$-curves; moreover, $L_{\widetilde{P}}\cdot \e_8=0$. 
\end{claim}

\begin{proof}[Proof of claim]
  We first consider $L_{\widetilde{R}}$. By Claim \ref{cl:1} we have that $L_{\widetilde{R}}-K_{\widetilde{R}} \equiv L_{\widetilde{R}}+T$ is big and nef. 
Therefore, if $|L_{\widetilde{R}}|$ fails to separate a scheme $Z$ of length $\leq 2$, then by Reider's theorem \cite[Thm. 1]{Rei} there exists an effective divisor $F$ containing $Z$ such that
\begin{equation} \label{eq:reider}
  \left(F \cdot (L_{\widetilde{R}}+T),F^2\right) \in \{(0,-1),(1,0),(0,-2),(1,-1),(2,0)\},
  \end{equation}
with the latter three occuring only if $\deg Z=2$. 
  We will show that the only possibility is the fourth one, with $F \cdot T=1$ and $F \cdot L_{\widetilde{R}}=0$.

To prove this, note that by \eqref{eq:df} the only negative curves in fibers $\f$ of $\widetilde{R}$ are the $(-1)$-curves
$\e_1,\e_9,\f-\e_1,\f-\e_9$, which have intersections
 \begin{eqnarray*}
  \e_1 \cdot (L_{\widetilde{R}}+T)=a_1+a_{10}+1 \geq 1, & \e_9 \cdot (L_{\widetilde{R}}+T)=a_0-a_{9}+1 \geq 1,\\
   (\f-\e_1)\cdot (L_{\widetilde{R}}+T)=a_0-a_{10}+1 \geq 1, &
(\f-\e_9)\cdot (L_{\widetilde{R}}+T)=a_1+a_9+1 \geq 1
\end{eqnarray*}
(using \eqref{eq:condcoff2}). Moreover, we have $T \cdot (L_{\widetilde{R}}+T) \geq 3$ by the above, and 
\begin{eqnarray*}
 \f \cdot (L_{\widetilde{R}}+T) & = & a_0+a_1+2 \geq 3, \\
  (s-\e_9)\cdot (L_{\widetilde{R}}+T) & = & a_1+\cdots+a_7+2a_9+2a_{10} \geq 3, \\
(s-\e_1)\cdot (L_{\widetilde{R}}+T) & = & a_0+a_2+\cdots+a_7+a_9+a_{10} \geq 3,
\end{eqnarray*}
using \eqref{eq:condcoff2}, \eqref{eq:condcoffJ1} and \eqref{eq:condcoffJ3}. All other curves $D\not \equiv \e_1,\e_9,\f-\e_1,\f-\e_9,\f,\s-\e_1,\s-\e_9, T$ intersect $\f$, $(s-\e_9)$ and $(s-\e_1)$ positively and $T$ nonnegatively, whence
\[
  D \cdot (L_{\widetilde{R}}+T) \geq D \cdot L_{\widetilde{R}} \geq a_0+a_1+\cdots +a_7+a_9+a_{10} \geq 3,
\]  
using again \eqref{eq:condcoffJ3}. This proves that the only possibility in \eqref{eq:reider} is $F^2=-1$, $F \cdot T=-F \cdot K_{\widetilde{R}}=1$ and $F \cdot L_{\widetilde{R}}=0$. In particular, it shows that $|L_{\widetilde{R}}|$ defines a morphism that is an embedding except for the contraction of $(-1)$-curves, as desired.

We then consider $L_{\widetilde{P}}$. We have
\begin{equation} \label{eq:LsuP}
  L_{\widetilde{P}} \sim a_2(\ell-\e_2)+\cdots +a_7(\ell-\e_7)+a_9\ell+a_{10}\ell.
  \end{equation}
In particular $L_{\widetilde{P}} \cdot \e_8=0$ and $|L_{\widetilde{P}}|$ defines a birational morphism that is an isomorphism outside finitely many contracted $(-1)$-curves. This follows e.g. from \eqref{eq:dp} and \cite[Prop. 3.10]{kn-dp}, as
$-L_{\widetilde{P}} \cdot K_{\widetilde{P}}=L_{\widetilde{P}} \cdot T =L_{\widetilde{R}} \cdot T \geq 5$ by Claim \ref{cl:1}.
\end{proof}

  We note that $T$ is nef on $\widetilde{P}$. As $L_{\widetilde{P}} \sim (L_{\widetilde{P}}+T)+K_{\widetilde{P}}$, we have $ h^j(L_{\widetilde{P}})=0, j=1,2$.
As $L_{\widetilde{R}}(-T) \equiv L_{\widetilde{R}}+K_{\widetilde{R}}$ we  have
$h^j(L_{\widetilde{R}}(-T))=0, j=1,2.$
  From the short exact sequence
  \begin{equation} \label{eq:dc}
  \xymatrix{
    0 \ar[r] & L_{\widetilde{R}}(-T) \ar[r] & L \ar[r] & L_{\widetilde{P}} \ar[r] & 0}
  \end{equation}
   and Riemann-Roch on ${\widetilde{R}}$ and ${\widetilde{P}}$ we therefore find that
  \begin{eqnarray*}
    h^0(L) & = & \chi(L)=\chi(L_{\widetilde{R}}(-T))+\chi(L_{\widetilde{P}})\\
    & = & \frac{1}{2}(L_{\widetilde{R}}+K_{\widetilde{R}})\cdot L_{\widetilde{R}}+\chi(\O_{\widetilde{R}})+\frac{1}{2}L_{\widetilde{P}}\cdot(L_{\widetilde{P}}-K_{\widetilde{P}}) +\chi(\O_{\widetilde{P}})\\
    & = & \frac{1}{2}(L_{\widetilde{R}}^2+L_{\widetilde{P}}^2)-\frac{1}{2}L_{\widetilde{R}} \cdot T + \frac{1}{2}L_{\widetilde{P}} \cdot T +\chi(\O_{\widetilde{R}})+\chi(\O_{\widetilde{P}}) \\
 & = & \frac{1}{2}L^2+ 1 =g  
\end{eqnarray*}
  (using the facts that $L_{\widetilde{R}} \cdot T =L_{\widetilde{P}} \cdot T$, $\chi(\O_{\widetilde{R}})=0$ and $\chi(\O_{\widetilde{P}})=1$). Furthermore, by \eqref{eq:dc} the restriction map
  $H^0(X,L) \longrightarrow H^0(\widetilde{P}, L_{\widetilde{P}})$ is surjective. Similarly, switching the roles of ${\widetilde{R}}$ and ${\widetilde{P}}$ (using
  that  $h^j(L_{\widetilde{R}})=0, j=1,2$, as $L_{\widetilde{R}}-K_{\widetilde{R}}
  \equiv L_{\widetilde{R}}+T$ is big and nef by Claim \ref{cl:1}, and
 $h^j(L_{\widetilde{P}}(-T))=h^j(L_{\widetilde{P}}+K_{\widetilde{P}})=0, j=1,2$), one finds that the restriction map $H^0(X,L) \longrightarrow H^0(\widetilde{R}, L_{\widetilde{R}})$ is surjective. Therefore, the morphism $\varphi_{L}$ defined by $|L|$ restricted to ${\widetilde{R}}$ and ${\widetilde{P}}$ is, respectively,
  the morphism defined by 
  $|L_{\widetilde{R}}|$ and $|L_{\widetilde{P}}|$. By Claim \ref{cl:2} the surfaces $\overline{R}:=\varphi_{L}(\widetilde{R})$ and $\overline{P}:=\varphi_{L}(\widetilde{P})$ are smooth.

  \begin{claim} 
    $\overline{R}$ and $\overline{P}$ intersect transversally and only along $\overline{T} :=\varphi_L(T) \cong T$
  \end{claim}

  \begin{proof}[Proof of claim]
    Assume first that $\O_{\overline{R}}(1)(-\overline{T})$ is globally generated.  
Assume that there is an intersection point $p$ of $\overline{R}$ and $\overline{P}$ outside $\overline{T}$. 
Let $C \subset \overline{P}$ be a general curve in $|\O_{\overline{P}}(1)|$ passing through $p$.
  Then $C$ intersects $\overline{T}$ transversally along a divisor $\xi \in |\O_{\overline{T}}(1)|$. The ideal sequence of
  $\xi \subset {\overline{T}} \subset \overline{R}$ tensored by $\O_{\overline{R}}(1)$:
  \begin{equation} \label{eq:lift}
    \xymatrix{ 0 \ar[r] & \O_{\overline{R}}(1)(-{\overline{T}}) \ar[r] & \O_{\overline{R}}(1) \* \I_{\xi/{\overline{R}}} \ar[r] &  \O_{{\overline{T}}} \ar[r] & 0,}
    \end{equation}
and the vanishing
$h^1(\O_{\overline{R}}(1)(-{\overline{T}}))=0$ (as $\O_{\overline{R}}(1)(-{\overline{T}}) \equiv \O_{\overline{R}}(1)(K_{\overline{R}})$)
prove  that $|\O_{\overline{R}}(1) \* \I_{\xi/{\overline{P}}}|$ is base point free off $\xi$, so that we can find a $C' \in |\O_{\overline{R}}(1)|$ not passing through $p$ and such that $C \cap C'=\xi$, a contradiction.

This proves that $\overline{R}$ and $\overline{P}$ do not intersect outside ${\overline{T}}$.
A similar argument shows that the intersection of $\overline{R}$ and $\overline{P}$ in $\PP^{g-1}$ is transverse: just replace $p$ with the infinitely near point of tangency of ${\overline{T}}$ at a supposed point of non-transversality.

If $\O_{\overline{P}}(1)(-{\overline{T}})$ is globally generated, we repeat the argument interchanging ${\overline{R}}$ and ${\overline{P}}$.

Finally we treat the case where neither $\O_{\overline{R}}(1)(-{\overline{T}})$ nor
$\O_{\overline{P}}(1)(-{\overline{T}})$ are globally generated. Since $\O_{\overline{R}}(1)(-\overline{T}) \equiv \O_{\overline{R}}(1)(K_{\overline{R}})$ and 
$\O_{\overline{R}}(1)^2=L_{\widetilde{R}}^2 \geq 5$ by Claim \ref{cl:1}, 
Reider's theorem \cite[Thm. 1]{Rei} yields the existence of an effective divisor $E_{\overline{R}}$ satisfying $E_{\overline{R}}^2=0$ and
$E_{\overline{R}} \cdot \O_{\overline{R}}(1)=1$. Thus, $E_{\overline{R}}  \cong \PP^1$, whence $E_{\overline{R}} \cdot K_{\overline{R}}=-2$. The total transform $E_{\widetilde{R}}$ of  $E_{\overline{R}}$ on $\widetilde{R}$ contains a smooth rational curve as a component, whence it must be supported on fibers of $\pi:\widetilde{R} \to E$. Since moreover $E_{\widetilde{R}}^2=0$, $E_{\widetilde{R}} \cdot K_{\widetilde{R}}=-2$, we must have $E_{\widetilde{R}} \equiv \f$. 
From \eqref{eq:effI} we find
$\f \cdot L_{\widetilde{R}}= a_0+a_1$, whence $(a_0,a_1) \in \{(1,0),(0,1)\}$. Moreover, the images of the fibers $\f$ by $\varphi_L$ are lines.

Assume $(a_0,a_1) =(0,1)$. Then \eqref{eq:condcoff2},  \eqref{eq:condcoffJ2} and  \eqref{eq:condcoffJ3} yield $a_9=a_{10}=0$, $a_3=a_4=1$, $a_2>0$, whereas $a_i \in \{0,1\}$ for $i \in \{5,6,7\}$. From \eqref{eq:LsuP} we find
$L_{\widetilde{P}} \sim (a_2+a_5+a_6+a_7+2)\ell-a_2\e_2-\e_3-\e_4-a_5\e_5-a_6\e_6-a_7\e_7$. But then one easily verifies that $\O_{\overline{P}}(1)(-{\overline{T}}) \sim
(a_2-1)(\ell-\e_2)+(a_5+a_6+a_7)\ell$, which is globally generated, a contradiction.

Assume $(a_0,a_1) =(1,0)$. Then conditions \eqref{eq:condcoff2},  \eqref{eq:condcoffJ2} and  \eqref{eq:condcoffJ3} yield $a_9=a_{10}=1$ and $a_3=\cdots=a_7=0$, whereas $a_2$ is arbitrary. From \eqref{eq:LsuP} we find
$L_{\widetilde{P}} \sim a_2(\ell-\e_2)+2\ell$. If $a_2>0$, then $\overline{P} \cong \Bl_{x_2}\PP^2$ and $\O_{\overline{P}}(1)(-{\overline{T}}) \sim
(a_2-1)(\ell-\e_2)$, which is globally generated, a contradiction.  We must therefore have $a_2=0$, in which case $\O_{\overline{P}}(1) \sim 2\ell$, that is, 
$\overline{P}$ is the $2$-uple embedding of $\PP^2$, and
$\O_{\overline{R}}(1) \equiv \s+3\f-\e_1$ by \eqref{eq:effI}. In particular, $g=6$.

Assume that there is an intersection point $p$ of $\overline{R}$ and $\overline{P}$ outside $\overline{T}$.  
Let $\mathfrak{l}_R \subset \overline{R}$ be the line in the ruling passing through $p$. This intersects ${\overline{T}}$ in two points, say $q$ and $q'$  (which are distinct, as the cover $\pi|_T:T \to E$ is {\'e}tale). Let $\mathfrak{c}_P \subset \overline{P}  \cong \PP^2$ be the  unique member of $|\O_{\PP^2}(1)|$  passing through $p$ and $q$, which is embedded as a conic in $\PP^{g-1} \cong \PP^5$. Then $\mathfrak{l}_R$ and
$\mathfrak{c}_P$ intersect in $p$ and $q$, and at no further points, for reason of degree.  In particular, $\mathfrak{c}_P$ does not contain $q'$. The intersection of the plane spanned by $\mathfrak{l}_R$ and
$\mathfrak{c}_P$ with $\overline{T}$ thus contains the length-three scheme $\mathfrak{c}_P \cap \overline{T}$ and $q'$. We therefore get  a $4$-secant plane to ${\overline{T}}$, a contradiction: indeed, the curve ${\overline{T}} \subset \PP^{5}$ has degree  $\overline{T} \cdot \O_{\overline{P}}(1)= 3\ell \cdot 2 \ell=6$;  thus, given any three points of ${\overline{T}}$, the system of hyperplanes through these three points cut out on ${\overline{T}}$ a complete linear series of degree at least  $3$,  and therefore has no base points.

Finally, assume that $p \in {\overline{T}}$  is a  point of non-transversality of $\overline{R} \cap \overline{P}$. Let again $\mathfrak{l}_R \subset \overline{R}$ be the line in the ruling passing through $p$, and denote by $q$  its further intersection point with ${\overline{T}}$, which is distinct from $p$ as above. Let $\mathfrak{c}_p \subset \overline{P}  \cong \PP^2$ (respectively, $\mathfrak{c}_{q}$) be a general member of $|\O_{\PP^2}(1)|$ passing through $p$ (resp., $q$), which is embedded as a conic in $\PP^{5}$. Then the intersections $\mathfrak{c}_p \cap \overline{T}$ and $\mathfrak{c}_{q} \cap \overline{T}$ each consist of three points, mutually distinct. 
If the intersection
$\mathfrak{c}_p \cap \mathfrak{l}_R$ is not  transversal  at $p$, then
$\mathfrak{c}_p$ and $\mathfrak{l}_R$ span a plane whose intersection with $\overline{T}$ contains the length-three scheme $\mathfrak{c}_p \cap \overline{T}$ and $q$. We thus again get a 
$4$-secant plane to $\overline{T}$, a contradiction as above. The same reasoning works for $\mathfrak{c}_{q}$. Thus, $\overline{C}:=\mathfrak{c}_p+\mathfrak{c}_{q}$ is a member of $\O_{\overline{P}}(1)$, that is, a hyperplane section of $\overline{P}$, intersecting
$\mathfrak{l}_R$ transversally in $p$ and $q$ and $\overline{T}$ transversally in a scheme $\xi$ consisting of $6$ distinct points (including $p$ and $q$). 
The exact sequence \eqref{eq:lift} above and the vanishings
$h^1(\O_{\overline{R}}(1)(-{\overline{T}}))=0$ (proved as above) and
$h^0(\O_{\overline{R}}(1)(-{\overline{T}}))=0$ (as $\O_{\overline{R}}(1)(-\overline{T}) \equiv -\s+4\f$) show that there is a unique member $\overline{D} \in |\O_{\overline{R}}(1)|$ passing through $\xi$ (necessarily containing $\mathfrak{l}_R$) and intersecting $\overline{T}$ transversally only along the six points in $\xi$, as $\overline{T} \cdot \O_{\overline{R}}(1)=6$. Thus, locally at $p$ we have $\overline{D} \cap \overline{C}= \mathfrak{c}_p \cap \mathfrak{l}_R$, which we proved to be transversal. Therefore, $\overline{D} \cup \overline{C}$ is a hyperplane section of $\overline{X}$ with an ordinary double point at $p$, a contradiction. 
 
\end{proof}

  The last two claims prove all assertions in the proposition except (i) and (ii).

   The proof of (i) follows the lines of the proof of
Lemma \ref{lemma:restcoho}.   

  To prove (ii), note that $\O_{\overline{T}}(K_{\overline{R}}+\overline{T}) \cong \omega_{\overline{T}} \cong \O_{\overline{T}}$, so that $(K_{\overline{R}}+{\overline{T}},0)$ is Cartier and represents the canonical divisor on $\overline{X}$. By \eqref{eq:T3}  we have
$K_{\overline{X}} \neq 0$ and $2K_{\overline{X}}=0$.  
\end{proof}

\begin{remark} \label{rem:chicont}
  Looking at \eqref{eq:effI} and \eqref{eq:LsuP} one can find precisely  which curves are contracted by $\varphi_L$ in terms of the coefficients $a_i$ and thus what $\overline{R}$ and $\overline{P}$ are, as well as their hyperplane bundles. Indeed, assume that $a_9 \geq a_{10}$ and $a_1 \geq a_2$. Then on $\widetilde{R}$ the curve $\e_1$ is contracted if and only if $a_i=0$ for all $i \neq 0,9$,
  the curve $\f-\e_1$ is contracted if and only if $a_0=a_9=a_{10}$, whereas $\e_9$ is contracted if and only if $a_0=a_9$. No other curves are contracted (note that $\f-\e_9$ cannot be contracted, because it would imply $a_1=a_9=0$, which is inconsistent with \eqref{eq:condcoff2}-\eqref{eq:condcoffJ1}). On $\widetilde{P}$ the curve $\e_i$ is contracted if and only if $a_i=0$, and no further curve except for $\e_8$ is contracted.  
\end{remark}

The first cotangent sheaf $T^1_{\overline{X}}$ of $\overline{X}$ sits in a short exact sequence
\begin{equation} \label{eq:cot} \xymatrix{
0 \ar[r] & \T_{\overline{X}} \ar[r] & \T_{\PP^{g-1}}|_{\overline{X}} \ar[r] & \N_{\overline{X}/\PP^{g-1}} \ar[r] & T^1_{\overline{X}} \ar[r] & 0, 
}
\end{equation}
where $\T_{\overline{X}}:=\Shom(\Omega_{\overline{X}},\O_{\overline{X}})$,
  and satisfies 
\begin{equation} \label{eq:ss2}
  T^1_{\overline{X}} \cong \N_{T/\overline{R}} \* \N_{T/\overline{P}} 
\end{equation}
by \cite[Prop. 2.3]{fri}. Note that by \eqref{eq:ss} we have that
$\deg T^1_{\overline{X}}$ equals the number of contracted curves by $\varphi_{L}$
and is therefore at least one, by Proposition \ref{prop:immersione}.

\begin{lemma} \label{lemma:cohnormal}
  Let $\overline{X}$ be as in Proposition \ref{prop:immersione}. Then 
  \begin{itemize}
  \item[(i)] $H^j(\overline{X},\N_{\overline{X}/\PP^{g-1}})=0$ for $j=1,2$; and
  \item[(ii)] the map $H^0(\overline{X},\N_{\overline{X}/\PP^{g-1}}) \to H^0({\overline{T}},T^1_{\overline{X}})$ induced by \eqref{eq:cot} is surjective.
    \end{itemize}
  \end{lemma}

\begin{proof}
  To prove the lemma we will argue much as in \cite[Pf. of Lemma 3]{clm}. We first have to deduce several vanishings of cohomology of sheaves on ${\overline{R}}$ and ${\overline{P}}$.

  Recall that $T^1_{\overline{X}}$ is supported on $\overline{T}$, where it is a line bundle of positive degree. Moreover, since ${\overline{T}} \equiv -K_{\overline{R}}$ and ${\overline{T}} \not \sim -K_{\overline{R}}$ on
  $\overline{R}$ and $K_{\overline{R}}^2 \leq 0$ as $\overline{R}$
  is a blow down of a blow up of $R$, we have
  \begin{equation} \label{eq:hjt1}
 h^j(T^1_{\overline{X}})=h^j(T^1_{\overline{X}} \*\O_{\overline{R}}(-{\overline{T}}))=0, \; \; j=1,2.
    \end{equation}

  We next claim that
  \begin{equation} \label{eq:h2tp}
H^2(\overline{P},\T_{\overline{P}})=0.
    \end{equation} 
 Indeed, for $f: S' \to S$ a blow up morphism of a single point of a smooth projective surface $S$ we have the dual of the sequence of relative differentials
  \[
  \xymatrix{ 0 \ar[r] & \T_{S'} \ar[r] & f^*\T_{S}  \ar[r] & \O_{\PP^1}(1) \ar[r] & 0,}
  \]
  which shows that $h^2(\T_{S'})=h^2(\T_{S})$. Thus \eqref{eq:h2tp} follows since $\overline{P}$ is birational to $\PP^2$.

  We then claim that
 \begin{equation} \label{eq:h2tr}
H^2(\overline{R},\T_{\overline{R}}(-{\overline{T}}))=0.
    \end{equation}   
Indeed, recall that $\overline{R}$ is obtained from $R:=\Sym^2(E)$ by a sequence of blow ups and downs of $(-1)$-curves always intersecting the strict transforms of $T$ in one point, as $T$ is numerically anticanonical at each step. As above, for  $f: S' \to S$ a blow up morphism of a single point of a smooth projective surface $S$ lying on a smooth curve $D \subset S$, letting $\e$ be the exceptional curve and $D' \sim f^*D -\e$ the strict transform of $D$ on $S$, we have
the  dual of the sequence of relative differentials twisted by $\O_{S'}(-D') \cong f^*\O_S(-D)\* \O_{S'}(\e)$:
 \[
  \xymatrix{ 0 \ar[r] & \T_{S'}(-D') \ar[r] & f^*(\T_{S}(-D))(\e)  \ar[r] & \O_{\PP^1} \ar[r] & 0,}
  \]
which shows that $h^2(\T_{S'}(-D'))=h^2(f^*(\T_{S}(-D))(\e))=h^2(\T_{S}(-D)\*f_*\O_{S'}(\e))= h^2(\T_{S}(-D))$. Thus \eqref{eq:h2tr} is equivalent to $h^2(\T_R(-T))=0$, which holds by Lemma \ref{lemma:g1}.

For simplicity we will in the rest of the proof denote by $\N_{\overline{X}}$, $\N_{\overline{R}}$ and $\N_{\overline{P}}$ the normal bundles of $\overline{X}$, ${\overline{R}}$ and ${\overline{P}}$ in $\PP^{g-1}$, respectively.

We claim that
 \begin{equation} \label{eq:hjnp}
H^j(\overline{P},\N_{\overline{P}})=0, \; \; j=1,2.
    \end{equation}   
 To prove this, consider the Euler sequence
  \[
  \xymatrix{ 0 \ar[r] & \O_{\overline{P}} \ar[r] & \O_{\overline{P}}(1)^{\+g}  \ar[r] & \T_{\PP^{g-1}}|_{\overline{P}} \ar[r] & 0.}
  \]
  Since $\O_{\overline{P}}(1) \sim K_{\overline{P}}+\O_{\overline{P}}(1)({\overline{T}})$ and ${\overline{T}}$ is nef (as $T^2 = 2$ on $\widetilde{P}$ and $\overline{P}$ is obtained from $\widetilde{P}$ by blowing down $(-1)$-curves), we get  $h^j( \O_{\overline{P}}(1))=0$ for $j=1,2$, whence $h^j(\T_{\PP^{g-1}}|_{\overline{P}})=0$ for $j=1,2$. Then \eqref{eq:hjnp} follows from  \eqref{eq:h2tp} and the normal bundle sequence
 \[
 \xymatrix{ 0 \ar[r] & \T_{\overline{P}} \ar[r] & \T_{\PP^{g-1}}|_{\overline{P}} \ar[r] & \N_{\overline{P}} \ar[r] & 0.}
  \]

 We then claim that
 \begin{equation} \label{eq:hjnr}
H^j(\overline{R},\N_{\overline{R}}(-{\overline{T}}))=0, \; \; j=1,2.
    \end{equation}   
To prove this, consider the Euler sequence twisted by $\O_{\overline{R}}(-{\overline{T}}))$:
  \[
  \xymatrix{ 0 \ar[r] & \O_{\overline{R}}(-{\overline{T}}) \ar[r] & \O_{\overline{R}}(1)(-{\overline{T}})^{\+g}  \ar[r] & \T_{\PP^{g-1}}|_{\overline{P}}(-{\overline{T}}) \ar[r] & 0.}
  \]
  Since $\O_{\overline{R}}(1)(-{\overline{T}}) \equiv \O_{\overline{R}}(1)(K_{\overline{R}})$ we get
  $h^j( \O_{\overline{R}}(1)(-{\overline{T}}))=0$ for $j=1,2$. Moreover 
  $h^2(\O_{\overline{R}}(-{\overline{T}}))=h^0(K_{\overline{R}}+{\overline{T}})=0$, as ${\overline{T}}$ is not anticanonical. Hence
$h^j(\T_{\PP^{g-1}}|_{\overline{R}}(-{\overline{T}}))=0$ for $j=1,2$. Then, from the normal bundle sequence twisted by $\O_{\overline{R}}(-{\overline{T}})$:
 \[
 \xymatrix{ 0 \ar[r] & \T_{\overline{R}}(-{\overline{T}}) \ar[r] & \T_{\PP^{g-1}}|_{\overline{R}}(-{\overline{T}}) \ar[r] & \N_{\overline{R}}(-{\overline{T}}) \ar[r] & 0,}
  \] 
and \eqref{eq:h2tr}, we obtain \eqref{eq:hjnr}.

We now prove that
 \begin{equation} \label{eq:hjnxap}
H^j(\overline{P},{\N_{\overline{X}}}|_{\overline{P}})=0, \; \; j=1,2.
    \end{equation}   
 To prove this, we use the short exact sequence
 \begin{equation} \label{eq:clmlemma2}
 \xymatrix{ 0 \ar[r] & \N_{\overline{P}} \ar[r] & {\N_{\overline{X}}}|_{\overline{P}}
   \ar[r] & T^1_{\overline{X}} \ar[r] & 0,}
  \end{equation} 
 whose existence is proved  in \cite[Lemma 2]{clm}. Then \eqref{eq:hjnxap} follows from  \eqref{eq:hjt1} and \eqref{eq:hjnp}.

 Similarly, using the sequence similar to \eqref{eq:clmlemma2} on $\overline{R}$
 twisted by $-{\overline{T}}$ and \eqref{eq:hjt1} and \eqref{eq:hjnr}, we obtain that
 \begin{equation} \label{eq:hjnxar}
H^j(\overline{R},{\N_{\overline{X}}}|_{\overline{R}}(-{\overline{T}}))=0, \; \; j=1,2.
    \end{equation}  

 We can now prove the lemma. From \eqref{eq:hjnxap}, \eqref{eq:hjnxar} and the short exact sequence
 \[
 \xymatrix{ 0 \ar[r] & {\N_{\overline{X}}}|_{\overline{R}}(-{\overline{T}}) 
\ar[r] & {\N_{\overline{X}}} 
   \ar[r] & {\N_{\overline{X}}}|_{\overline{P}}\ar[r] & 0,}
  \]
  we get $h^j(\N_{\overline{X}})=0$ for $j=1,2$, proving (i), as well as the restriction map
  $H^0(\overline{X},\N_{\overline{X}}) \to  H^0(\overline{P},{\N_{\overline{X}}}|_{\overline{P}})$ being surjective. Since $H^0(\overline{P},{\N_{\overline{X}}}|_{\overline{P}})$ surjects onto $H^0(T^1_{\overline{X}})$ by \eqref{eq:clmlemma2} and \eqref{eq:hjnp}, we see that $H^0(\overline{X},\N_{\overline{X}})$ surjects onto $H^0(T^1_{\overline{X}})$, proving (ii). 
  \end{proof}

\begin{thm} \label{thm:hilbert}
Let $\overline{X}$ be as in Proposition \ref{prop:immersione}. Then $\overline{X}$ is
represented by a smooth point $[\overline{X}]$ of the Hilbert scheme parametrizing surfaces of degree $2g-2$ in $\PP^{g-1}$. The irreducible component $\mathfrak{H}$ containing $[\overline{X}]$ 
is reduced and has dimension $g^2+9$ and its general point parametrizes a smooth Enriques surface $S$.
\end{thm}

\begin{proof}
  Since $H^1(\N_{\overline{X}/\PP^{g-1}})=0$ by Lemma \ref{lemma:cohnormal}(i), the
  point $[\overline{X}]$ representing $\overline{X}$ in the Hilbert scheme of $\PP^{g-1}$ is smooth \cite[Thm. 4.3.5]{ser}
and thus 
belongs to a single reduced component $\mathfrak{H}$ of it. By Lemma \ref{lemma:cohnormal}(ii), a general tangent vector to $\mathfrak{H}$ at $[\overline{X}]$   represents a first-order embedded deformation of $\overline{X}$ that smooths the double curve ${\overline{T}}$. Hence the general point in $\mathfrak{H}$ represents a smooth irreducible surface $S$.
  Since $h^j(\O_{\overline{X}})=0$ for $j=1,2$ by  Proposition \ref{prop:immersione}(i), also $h^j(\O_S)=0$ for $j=1,2$. Moreover $2K_{\overline{X}}=0$ by Proposition \ref{prop:immersione}(ii), whence $K_S \equiv 0$.  It follows that $S$ is a smooth Enriques surface.

  Since $[S]$ is a smooth point of $\mathfrak{H}$, we  have $\dim \mathfrak{H}=h^0(\N_{S/\PP^{g-1}})$, which can be computed using $h^0(\T_S)=h^2(\T_S)=0$ and $h^1(\T_S)=10$, the
 normal bundle sequence
 \[
 \xymatrix{ 0 \ar[r] & \T_{S} \ar[r] & \T_{\PP^{g-1}}|_{S} \ar[r] & \N_{S/\PP^{g-1}} \ar[r] & 0,}
  \]  
and the Euler sequence for $S \subset \PP^{g-1}$.
\end{proof}

\begin{corollary} \label{cor:piclimite}
  With the same assumptions as in Proposition \ref{prop:immersione}, there is a flat family $\pi:\mathfrak{X} \to \DD$ over the unit disc such that $\mathfrak{X}$ is smooth, $\pi^{-1}(0)=X$ and $S_t:=\pi^{-1}(t)$ is a smooth Enriques surface for $t \neq 0$, and a line bundle $\L$ on $\mathfrak{X}$ such that $\L|_{X}=L$ and $\L|_{S_t}$ is very ample for $t \neq 0$.

Furthermore, there is a short exact sequence
\begin{equation} \label{eq:piclimite}
  \xymatrix{
      0 \ar[r] & \ZZ[\xi] \ar[r] & \Pic X \cong H^2(\mathfrak{X},\ZZ) \ar[r]^{\iota_t^*} & H^2(S_t,\ZZ) \cong \Pic S_t  \ar[r] & 0,}
  \end{equation}
where $\iota_t: S_t \subset \mathfrak{X}$ is the inclusion. 
\end{corollary}

\begin{proof}
  As mentioned above, $\deg T^1_{\overline{X}}$ equals the number of contracted curves by $\varphi_{L}$ and there is a section $s \in H^0(T^1_{\overline{X}})$ such that the support of its zero scheme $Z(s)$ is precisely the  images of the contracted curves. By Lemma \ref{lemma:cohnormal}(ii) this section can be lifted to a section of $H^0(\overline{X},\N_{\overline{X}/\PP^{g-1}})$, which in turn defines an embedded deformation of $\overline{X}$. Let $p:\mathfrak{X'} \to \DD$ be the universal family. Then  $\mathfrak{X}'$ is singular precisely along $Z(s)$, cf., e.g.,  
\cite[Chp. 2]{ser} or \cite[\S 2]{fri}, and in particular the general fiber of $p$ is a smooth Enriques surface.  The singularities of $\mathfrak{X}'$ can be resolved by a small resolution in the following way, cf. e.g. \cite[p.~647]{clm}: the tangent cone to $\mathfrak{X}'$ at each of the singular points has rank $4$. The exceptional divisors of the blow up $\widetilde{\mathfrak{X}} \to \mathfrak{X}'$ at these points are rank $4$ quadric surfaces. These can be contracted along any of the two rulings on one of the two irreducible components of the strict transform of $\overline{X}$ in $\widetilde{\mathfrak{X}}$ by a contraction map $\widetilde{\mathfrak{X}} \to {\mathfrak{X}}$. One obtains a morphism $\mathfrak{X} \to \mathfrak{X}'$, which is the desired small resolution, and one can make sure that the central fiber is $X$ by choosing
which of the components of the central fiber to contract along.

  Since the sum of the irregularities of the components of $X$ equals the genus of the double curve of $X$, and $b_1(S_t)=0$ for $t \neq 0$,
 we have an exact sequence
\[ \xymatrix{
  0 \ar[r] & H^0(S_t,\QQ)   \ar[r] & H_4(\mathfrak{X},\QQ) \ar[r] \ar@{}[d]|*=0[@]{\cong} & H^2(\mathfrak{X},\QQ) \ar[r]^{{\iota_t}^*} \ar@{}[d]|*=0[@]{\cong} & H^2(S_t,\QQ) \ar[r] & 0, \\
  & &H_4(X,\QQ)& H^2(X,\QQ)  &&}
    \]
    cf. \cite[\S 4(c)]{M2}.
    Since $H^0(S_t,\ZZ) \cong \ZZ$  and  $H_4(X,\ZZ) \cong \ZZ^2$,  we obtain an exact sequence
\[ \xymatrix{
  0 \ar[r] & \ZZ \ar[r] & H^2(X,\ZZ)/\mathfrak{T} \ar[r]^{i_t^*} & H^2(S_t,\ZZ)/\mathfrak{T}_t \ar[r] & 0,}
  \]
  where $\mathfrak{T} \subset H^2(X,\ZZ)$ and
  $\mathfrak{T}_t \subset H^2(S_t,\ZZ)$ are the torsion subgroups. 
  We have  $\Pic S_t \cong H^2(S_t,\ZZ)$  and $\mathfrak{T}_t\cong \ZZ_2[K_{S_t}]$,  and 
also $\Pic X \cong H^2(X,\ZZ)$ by  Lemma \ref{lemma:restcoho} and 
  $\mathfrak{T} \cong \ZZ_2[K_X]$ by Lemma \ref{lemma:onlytor}. Since 
    $[\xi] \in \ker{{\iota_t}^*}$  and $\xi$ is indivisible,  the sequence \eqref{eq:piclimite} follows.  
\end{proof}

By the latter result we may extend the notions of isotropic $10$-sequence, of simple isotropic decompositions
and of admitting the same simple  decomposition type (modulo $\xi$) to all members of $\D$.

The next result is crucial in the proof of Theorem \ref{mainthm1}.

\begin{proposition} \label{prop:main}
  Let $(X,L)$ be as in Proposition \ref{prop:immersione} and assume that one of the following holds:
  \begin{itemize}
  \item[(i)] $a_0$ is odd, $a_9$ is even;
  \item[(ii)] $a_0$ is even and nonzero, $a_9$ is odd;
\item[(iii)] $a_0,a_9,a_{10}$ are odd, $a_1$ is even;
  \item[(iv)] $a_i$ is odd for some $i \in \{2,\ldots,7\}$.
    \end{itemize}
 Then there is a flat family $f:\mathfrak{X} \to E^{\circ}$, where $E^{\circ} \subset E$ is a Zariski-open dense subset,  parametrizing  surfaces in $\D$,
  a line bundle $\L$ on $\mathfrak{X}$ and points $t_0, t_0 \+ \eta \in E^{\circ}$ such that
  \begin{itemize}
  \item $f^{-1}(t_0)=f^{-1}(t_0\+\eta)=X$,
    \item $\L|_{f^{-1}(t_0)} \cong L$ and $\L|_{f^{-1}(t_0\+\eta)} \cong L+K_X$.
    \end{itemize}
\end{proposition}

\begin{proof}
 {\bf Case (i).} 
By \eqref{eq:duesez} we have a double cover $\Psi:E \to T$ mapping $t \in E$ to $\s_{t} \cap T=\{t+(t\+\eta)\}$, which identifies $t$ with $t \+ \eta$.

Pick seven general distinct points $x_1,\ldots, x_7$ on $T$.
For $t \in E$, set $x_9^{t}:=\Psi(t)$ and let $x_8^{t}$ be the unique point on $T$ such that $x_1+x_2+\cdots+x_7+x_8^{t}+x_9^{t} \in |\N_{T/R} \* \N_{T/\PP^2}|$.
Let
$\widetilde{P}_{t}:=\Bl_{x_2,\ldots,x_7,x_8^{t}}\PP^2$,
$\widetilde{R}_{t}:=\Bl_{x_1,x_9^{t}}R$ and $X_{t}:=\widetilde{R}_{t} \cup_T \widetilde{P}_{t}$, obtained by the obvious gluing. Then
there is a Zariski-open dense subset $E^{\circ}$ of $E$ such that
$\{X_{t}\}_{t \in E^{\circ}}$ is a flat family  of  surfaces in $\D$, with $X_{t}=X_{t\+\eta}$, since $x_9^{t}=\Psi(t)=\Psi(t\+\eta)$. Let $X=X_{t_0}$ for some $t_0 \in E^{\circ}$, with $x_8^{t_0}=x_8$ and $x_9^{t_0}=x_9$.
As $t_0$ deforms to $t_0 \+ \eta$, the surface $X=X_{t_0}$ deforms nontrivially back to itself. The divisor $\s_{t_0}$ deforms to $\s_{t_0\+\eta}$, whence $E_{9,10}=(\s_{t_0}-\e_9,0)$ deforms to $E_{9,10}+K_X=(\s_{t_0\+\eta}-\e_9,0)$, cf. \eqref{eq:110}. 
The divisor $E_9$ depends on $x_9$, and under this process it may deform either to
itself or to $E_9+K_X$, as $\f_{\alpha_9}$ may deform to itself or $\f_{\alpha_9 \+ \eta}$. Since $a_9$ is even, $a_9E_9$ will in any case deform back to itself. All other divisors
present in the decomposition \eqref{dec:immersione} are independent of both $x_8$ and $x_9$, thus remain invariant. Hence $L$ deforms to $L+K_X$ as $t_0$ deforms to $t_0 \+ \eta$.

{\bf Case (ii).} 
For each $t \in E$, let $\ell_{t}$ be the unique member of $|\ell|$ 
through the two points in $\f_{t}\cap T$.
We have a morphism $\Phi:E \to T$ mapping $t \in E$ to the residual intersection point of  $\ell_{t}$ with $T$; in other words
$\Phi(t)=(\ell_{t} \cap T) \setminus (\f_{t}\cap T)$. We note that $\Phi$ is a double cover, identifying each $t \in E$ with $t \+ \eta$. 
A key observation is that, by contrast, $\ell_{t} \neq \ell_{t \+ \eta}$.

Pick seven general distinct points $x_1,\ldots, x_7$ on $T$.
For any $t \in E$, set $x_9^{t}:=\Phi(t)$ and let $x_8^{t}$ be the unique point on $T$ such that $x_1+x_2+\cdots+x_7+x_8^{t}+x_9^{t} \in |\N_{T/R} \* \N_{T/\PP^2}|$.  Let
$\widetilde{P}_{t}:=\Bl_{x_2,\ldots,x_7,x_8^{t}}\PP^2$,
$\widetilde{R}_{t}:=\Bl_{x_1,x_9^{t}}R$ and $X_{t}=\widetilde{R}_{t} \cup_T \widetilde{P}_{t}$. Then there is a Zariski-open dense subset $E^{\circ}$ of $E$ such that $\{X_{t}\}_{t \in E^{\circ}}$ is a flat family  of  surfaces in $\D$, with $X_{t}=X_{t\+\eta}$, since $x_9^{t}=\Phi(t)=\Phi(t\+\eta)$. Let  $X=X_{t_0}$ for some $t_0 \in E^{\circ}$, with $x_8^{t_0}=x_8$ and $x_9^{t_0}=x_9$.
As $t_0$ deforms to $t_0 \+ \eta$, the surface $X=X_{t_0}$ deforms nontrivially back to itself. The divisor $E_9$ is on
$X$ represented by the pair $(\f_{t_0}+\e_9, \ell_{t_0})$. As $t_0$ deforms to $t_0 \+ \eta$, this will deform to $(\f_{t_0\+\eta}+\e_9, \ell_{t_0\+\eta}) \sim E_9+K_X$. 

Since $a_0$ is even, $a_0E_{9,10}$ will deform back to itself. All other divisors
present in the decomposition \eqref{dec:immersione} are independent of $x_8$ and $x_9$, thus remain invariant. It follows that $L$ deforms to $L+K_X$ as $t_0$ deforms to $t_0 \+ \eta$.

{\bf Case (iii).} 
By \eqref{eq:ijk} we have $E_{9,10}+E_{10} \sim E_{1,9}+E_{1}$, so that we may write
\begin{equation}
  \label{eq:nuovoII}
    L \sim (a_0-1)E_{9,10}+(a_1+1)E_1+a_2E_2+\cdots+a_7E_7+a_9E_9+(a_{10}-1)E_{10}+E_{1,9}.
\end{equation}
We note that no isotropic divisor present in this decomposition depends on $x_8$ and that the only ones depending on $x_1$ are $E_1$ and $E_{10}$, the first occurring with odd coefficient and the latter occurring with even coefficient in \eqref{eq:nuovoII}.

We now argue as in case (i), with $x_9$ replaced by $x_1$, and obtain a (nontrivial) deformation of $X$ back to itself in such a way that $E_1=(\s_{t_0}-\e_1,0)$ deforms to $E_1+K_X=(\s_{t_0\+\eta}-\e_1,0)$. Since $E_{10}$ occurs with even coefficients in the decomposition  \eqref{eq:nuovoII} and all other isotropic divisors in the decomposition  are independent of $x_1$ and $x_8$,  we see that $L$ deforms to $L+K_X$.

{\bf Case (iv).} 
We argue as in  case (ii), with $x_9$ replaced by $x_i$, and obtain a (nontrivial) deformation of $X$ back to itself in such a way that $E_i=(\f_{t_0},\ell-\e_i)$ deforms to $E_i+K_X$, and since all other isotropic divisors in the decomposition \eqref{dec:immersione} are independent of both $x_i$ and $x_8$, we see that $L$ deforms to $L+K_X$. 
\end{proof}

\section{Proof of Theorem \ref{mainthm1}} \label{sec:proof}

We will need the following result:

\begin{lemma} \label{lemma:sid}
   Let $L$ be any effective line bundle on an Enriques surface $S$ such that $L^2>0$, $\phi(L) \geq 3$ and $L$ is not numerically $2$-divisible. Then there is an isotropic $10$-sequence $\{E_1,\ldots,E_{10}\}$ on $S$  such that
   \[
     L \sim a_1E_1+\cdots +a_7E_7+a_9E_9+a_{10}E_{10}+a_0E_{9,10},
\]
   where 
   $a_0,a_1,\ldots,a_{10}$ are nonnegative integers satisfying conditions \eqref{eq:condcoff2}-\eqref{eq:condcoffJ3} (in Proposition \ref{prop:immersione}) and such that one of the conditions (i)-(iv) in Proposition \ref{prop:main} holds. 
\end{lemma}

  \begin{proof}
    Write $L$ as in \eqref{eq:scrivoL} satisfying the conditions of Proposition \ref{prop:sid}. Thus, \eqref{eq:condcoff2} and \eqref{eq:condcoffJ2} hold.
    Moreover, if $a_0=0$, then also $a_9=a_{10}=0$ by \eqref{eq:condcoff'}, whence $a_1,a_2>0$ since $L^2>0$. It follows that also \eqref{eq:condcoffJ1} holds.
    Finally, since $\phi(L) \geq 3$, we must have $E_i \cdot L \geq 3$ for
    $i \in \{1,2\}$, yielding \eqref{eq:condcoffJ3}.
Note that conditions \eqref{eq:condcoff2}-\eqref{eq:condcoffJ3} are symmetric with respect to interchanging $a_9$ and $a_{10}$, as well as $a_1$ and $a_2$.

    By \eqref{eq:defeps} and Remark \ref{rem:eps}, at least one of the coefficients $a_i$ is odd.

    Assume  that $a_0>0$.
    If $a_0$ and $a_9$ have different parities, we are in case (i) or (ii). Similarly, if $a_0$ and $a_{10}$ have different parities, we may interchange $E_9$ and $E_{10}$ and end up in case (i) or (ii). If $a_{i}>0$ is odd for some $i \in \{1,\ldots,7\}$, we end up in case (iv), possibly after 
    interchanging $E_1$ and $E_2$. The latter is the case if all $a_0,a_9,a_{10}$ are even. Left is the case where $a_0,a_9,a_{10}$ are odd and $a_1,\ldots,a_7$ are even, which yields case (iii).

    Assume  that $a_0=0$. Since $a_{i}>0$ is odd for some $i \in \{1,\ldots,7\}$, we end up in case (iv) possibly after interchanging $E_1$ and $E_2$. 
\end{proof}

\begin{proof}[Proof of Theorem \ref{mainthm1}]
  Assume that $\widehat{\E}'$ is an irreducible component of $\widehat{\E}_{g}$ parametrizing pairs $(S,[H])$ with $[H]$ $2$-divisible. Then by \cite[Lemma 4.8]{cdgk} the pairs $(S,H)$ and $(S,H+K_S)$ lie in different irreducible components of $\E_{g}$. Hence, $\rho_g^{-1}(\widehat{\E}')$ is reducible.

  Assume on the contrary that $\widehat{\E}'$ parametrizes
   numerically polarized surfaces with  classes that are not numerically $2$-divisible.  By \cite[Prop. 4.16]{cdgk} the component $\widehat{\E}'$ consists precisely of pairs admitting the same simple  decomposition type. Let $\phi$ be the $\phi$-value of the members of $\widehat{\E}'$. Assume first that $\phi \geq 3$. Then by  Lemma \ref{lemma:sid} the members of $\widehat{\E}'$ 
admit  the same simple  decomposition type (modulo $\xi$) as $(X,L)$, with $X$ in $\D$, 
  and $L  \sim a_1E_1+\cdots +a_7E_7+a_9E_9+a_{10}E_{10}+a_0E_{9,10}$ in $\Pic X$ satisfying the conditions in Propositions \ref{prop:immersione} and \ref{prop:main}.
By Proposition \ref{prop:immersione}, the line bundles $L$ and $L+K_X$ define morphisms $\varphi_L:X \to \PP^{g-1}$ and $\varphi_{L+K_X}:X \to \PP^{g-1}$, respectively, that are isomorphisms except for contractions of $(-1)$-curves on either component of $X$.
By Proposition \ref{prop:main}, the surfaces $\varphi_L(X)$ and $\varphi_{L+K_X}(X)$ lie in the same irreducible component of the Hilbert scheme $\mathfrak{H}$, and they are both smooth points of $\mathfrak{H}$ by Theorem \ref{thm:hilbert}.  

  Let $\pi:\mathfrak{X} \to \DD$ be the one-parameter family with parameter $t$ over the disc $\DD$
  of Corollary \ref{cor:piclimite}, with special fiber $\pi^{-1}(0)=X$ and general fiber a smooth Enriques surface $S_t=\pi^{-1}(t)$. Let $\iota_t: S_t \subset \mathfrak{X}$ be the inclusion. Using Corollary \ref{cor:piclimite} and the notation therein, and setting
  $E_i^{(t)}:=\iota_t^*E_i \in \Pic S_t$
and, similarly, $E_{9,10}^{(t)}:=\iota_t^*E_{9,10}$,
we get
\[ L_t:=\iota_t^*L \sim a_1E_1^{(t)}+\cdots +a_7E_7^{(t)}+a_9E_9^{(t)}+a_{10}E_{10}^{(t)}+a_0E_{9,10}^{(t)}.\]
In particular, $(S_t,L_t)$ admits the same  simple decomposition type as $(X,L)$. Moreover, $\varphi_{L_t}(S_t)$ also lies in $\mathfrak{H}$. Since $\iota_t^*K_X =K_{S_t}$, we have
\[ \iota_t^*(L+K_X) \sim a_1E_1^{(t)}+\cdots +a_7E_7^{(t)}+a_9E_9^{(t)}+a_{10}E_{10}^{(t)}+a_0E_{9,10}^{(t)}+K_{S_t} \sim L_t+K_{S_t},\]
and $\varphi_{L_t+K_{S_t}}(S_t)$ also lies in $\mathfrak{H}$. Thus, 
$(S_t,L_t)$ and $(S_t,L_t+K_{S_t})$ belong to the same irreducible component of $\E_{g}$.  By construction,  $\rho_g([S_t,L_t])=\rho_g([S_t,L_t+K_{S_t}])$ lies in $\widehat{\E}'$. Thus 
$\rho_g^{-1}(\widehat{\E}')$, containing both $(S_t,L_t)$ and $(S_t,L_t+K_{S_t})$,  is 
irreducible.

If  $\phi \leq 2$,  then one can  repeat the same reasoning substituting the  pairs $(S,[H])$ parametrized by $\widehat{\E}'$ with $(S,3[H])$.  Since Theorem \ref{mainthm1} in these cases follows from \cite[Cor. 1.3]{cdgk} anyway, we leave the details to the reader.
 \end{proof}

As a consequence of Theorem \ref{mainthm1}, we get a positive answer to \cite[Question 4.17]{cdgk}:

\begin{thm} \label{mainthm2}
  The irreducible components of $\E_{g}$ are precisely the loci parametrizing pairs admitting the same simple  decomposition type. 
\end{thm}

\begin{proof}
  By \cite[Prop. 4.16]{cdgk} the irreducible components of $\widehat{\E}_{g}$ are precisely the loci para-\linebreak metrizing pairs admitting the same simple  decomposition type modulo the canonical bundle. By Theorem \ref{mainthm1} the map $\rho_g$ gives a one-to-one correspondence between the irreducible components of $\E_{g}$ and $\widehat{\E}_{g}$ parametrizing pairs with line bundles that  are not numerically $2$-divisible. Moreover, in this case pairs of the form $(S,L)$ and $(S,L+K_S)$ admit the same simple  decomposition type by \cite[Cor. 4 and Rem. 4.11]{cdgk}: more precisely, this follows as one may always find a simple isotropic decomposition 
$L \sim \sum_{i=1}^n a_iE_i+\varepsilon K_S$ with at least one odd coefficient $a_{i_0}$ 
(by \cite[Lemma 4.8]{cdgk}),  and then setting $E'_{i_0}:=E_{i_0}+K_S$ one reaches a simple isotropic decomposition $L+K_S \sim \sum_{i=1, i \neq i_0}^n  a_iE_i+a_{i_0}E'_{i_0}+\varepsilon K_S$, which is of the same type as the one for $L$.
The theorem is therefore proved for components of $\E_{g}$ parametrizing classes that are not numerically $2$-divisible.

On the other hand, if $\E'$ is an irreducible component of $\E_{g}$ parametrizing classes $(S,L)$ that are numerically $2$-divisible, then $\rho_g^{-1}\rho_g(\E')$ consists of two irreducible components by Theorem \ref{mainthm1}, containing $(S,L)$ and $(S,L+K_S)$, respectively. Since the coefficients in any simple isotropic decomposition of $(S,L)$ and of $(S,L+K_S)$ are even by \cite[Lemma 4.8]{cdgk}, we see that $(S,L)$ and $(S,L+K_S)$ do not admit the same simple  decomposition type by \cite[Cor. 4 and Rem. 4.11]{cdgk}: indeed, assuming for instance $L=2M$ in $\Pic S$, then $L$ always has simple isotropic decompositions with $\varepsilon=0$, whereas  $L+K_S$ always has simple isotropic decompositions with $\varepsilon=1$. This proves the theorem for components of $\E_{g}$ parametrizing classes that are  numerically $2$-divisible.
 \end{proof}

\section{The $\phi$-vector and proofs of Theorems \ref{thm:phivector1} and \ref{mainthm3}} \label{sec:phi}

Recalling Proposition \ref{prop:sid}, we make the following:

\begin{definition} \label{def:fund}
Let $L$ be any effective line bundle on an Enriques surface $S$ such that $L^2>0$. A decomposition of the form \eqref{eq:scrivoL} with coefficients satisfying  \eqref{eq:defeps}, \eqref{eq:condcoff} and \eqref{eq:condcoff'}  is called a {\em fundamental presentation of $L$}. 
  \end{definition}

In the next two lemmas we  deduce some properties of isotropic $10$-sequences satisfying the conditions of Proposition \ref{prop:sid}, that is, appearing in fundamental presentations.

\begin{lemma} \label{lemma:intminime1}
  Let $\{E_1,\ldots,E_{10}\}$ be any isotropic $10$-sequence satisfying the conditions of Proposition \ref{prop:sid} and $F$ be any effective primitive isotropic divisor such that $F \not \equiv E_i$ for all $i \in \{1,\ldots,9\}$ and $F \not \equiv E_{9,10}$. Then
  \[ \phi(L)=E_1 \cdot L \leq E_2 \cdot L \leq \cdots \leq E_8 \cdot L \leq \min\{E_9 \cdot L, E_{9,10}\cdot L\} \leq   E_9 \cdot L \leq F \cdot L.\]
\end{lemma}

\begin{proof}
  Set $a_8:=0$ and $a:=\sum_{i=0}^{10}a_i$. Since $a_1 \geq \cdots \geq a_7 \geq a_8=0$, and $E_i \cdot L=a-a_i$ for all $i \in \{1,\ldots,8\}$, we get $E_1 \cdot L \leq \cdots \leq E_8 \cdot L=a$. Moreover, $E_9 \cdot L=a+a_0-a_9 \geq a$ since
  $a_0 \geq a_9$ and $E_{9,10} \cdot L=a+a_9+a_{10}-a_0 \geq a$ since
  $a_9+a_{10} \geq a_0$.

  Let now $F$ be an effective primitive isotropic divisor such that $F \not \equiv E_i$ for all $i \in \{1,\ldots,9\}$ and $F \not \equiv E_{9,10}$.

If $F \equiv E_{10}$, then $F \cdot L=a+a_0-a_{10} \geq a+a_0-a_9=E_9 \cdot L$ since
$a_{9} \geq a_{10}$.

If $F \not \equiv E_{10}$, then $F \cdot E_i>0$ for all $i \in \{1,\ldots,10\}$ (cf. \cite[Lemma 2.1]{klvan}), whence $F \cdot (E_1+\cdots+E_{10}) \geq 10$. But $E_1+\cdots+E_{10} \sim 3D$, with $D \sim E_9+E_{10}+E_{9,10}$ (cf. Lemma \ref{lemma:ceraprima}), whence $F \cdot (E_9+E_{10}+E_{9,10}) \geq 4$. Hence,
$F \cdot E_9 \geq 2$, or $F \cdot E_{10} \geq 2$, or $F \cdot E_{9,10} \geq 2$.

If $F \cdot E_9 \geq 2$, then $F \cdot L \geq a+a_9 \geq a+a_0-a_{10} \geq a+a_0-a_9=E_9 \cdot L$, using the facts that $a_9+a_{10}\geq a_0$ and $a_9 \geq a_{10}$
(and $F \cdot E_{9,10} >0$ by \cite[Lemma 2.1]{klvan}).

If $F \cdot E_{10} \geq 2$, then similarly $F \cdot L \geq a+a_{10} \geq a+a_0-a_9=E_9 \cdot L$.

If $F \cdot E_{9,10} \geq 2$, then  $F \cdot L \geq a+a_0 \geq a+a_0-a_{9} =E_9 \cdot L$.

We have therefore proved that $F \cdot L \geq E_9 \cdot L$. It follows that $E_1 \cdot L=\phi(L)$.
 \end{proof}

 \begin{remark} \label{rem:intminime1}
 For $a:=\sum_{i=0}^{10}a_i$ the last proof yields  $E_8 \cdot L =a$. Hence, for any  effective primitive isotropic divisor such that $F \not \equiv E_i$ for all $i \in \{1,\ldots,8\}$, we get $F \cdot L \geq a$. 
 \end{remark}

 \begin{lemma} \label{lemma:intminime2}
   Let $\{E_1,\ldots,E_{10}\}$ be any isotropic $10$-sequence satisfying the conditions of Proposition \ref{prop:sid}. For any isotropic $10$-sequence
   $\{F_1,\ldots,F_{10}\}$ we have $(F_1+\cdots+F_{10}) \cdot L \geq (E_1+\cdots+E_{10}) \cdot L$. If equality holds and $F_1 \cdot L \leq F_2 \cdot L \leq \cdots \leq F_{10} \cdot L$, then there is no $n \in \{1,\ldots,9\}$ such that $E_i \cdot L = F_i \cdot L$ for all $i \in \{1,\ldots,n-1\}$ and $F_n \cdot L < E_n \cdot L$.
 \end{lemma}

 \begin{proof}
  As in the previous proof, set $a_8:=0$ and $a:=\sum_{i=0}^{10}a_i$.  Let $D:=\frac{1}{3}(E_1+\cdots+E_{10}) \sim E_9+E_{10}+E_{9,10}$. Then
  \[ D \cdot L = (E_9+E_{10}+E_{9,10}) \cdot L = (a+a_0-a_9)+(a+a_0-a_{10})+(a+a_9+a_{10}-a_0)=3a+a_0.\]
  Let  $\{F_1,\ldots,F_{10}\}$ be any isotropic $10$-sequence. Set $D':=\frac{1}{3}(F_1+\cdots+F_{10}) \sim F_i+F_j+F_{i,j}$. We will prove that $D' \cdot L \geq 3a+a_0$. We may and will assume that $D' \not \equiv D$.

  We will divide the treatment into the two cases
  \begin{itemize}
  \item[(I)] $E_{9,10} \equiv F_i$ for some $i \in \{1,\ldots,10\}$.
  \item[(II)] $E_{9,10} \not \equiv F_i$ for all $i \in \{1,\ldots,10\}$.
  \end{itemize}

  {\bf Case (I).} Assume without loss of generality that $E_{9,10} \equiv F_1$.
  We have $D' \sim F_1 + F_j+F_{1,j}$ for each $j \in \{2,\ldots,10\}$. Then
  $F_j \not \equiv E_i$ for $i  \in \{9,10\}$ (as $F_j \cdot E_{9,10}=F_j \cdot F_1=1$). Hence there must exist $j \in \{2,\ldots,10\}$ such that $F_j \not \equiv E_i$ for all $i \in \{1,\ldots,10\}$. Assume without loss of generality that $j=2$. Note that $F_{1,2}  \not \equiv E_i$ for all $i \in \{1,\ldots, 8\}$ (as
  $F_{1,2} \cdot E_{9,10}=F_{1,2} \cdot F_1=2$).
  We divide into the cases
  \begin{itemize}
  \item[(I-i)] $F_{1,2} \equiv E_{k}$ for $k=9$ or $10$.
      \item[(I-ii)] $F_{1,2} \not \equiv E_i$ for all $i \in \{1,\ldots, 10\}$.
   \end{itemize}

   {\bf Case (I-i).} We have $F_2 \cdot E_k=F_2 \cdot F_{1,2}=2$ for $k=9$ or $10$ and $F_2 \cdot E_i >0$ for all $i \in \{1,\ldots,10\}$, as $F_2 \not \equiv E_i$, by \cite[Lemma 2.1]{klvan}. Thus $F_2 \cdot L \geq a+a_k$, whence
   \begin{eqnarray*}
     D' \cdot L & = & (F_1 + F_2+F_{1,2}) \cdot L = E_{9,10}\cdot L+F_2 \cdot L + E_k \cdot L \\
     & \geq & (a+a_9+a_{10}-a_0)+(a+a_{k})+(a+a_0-a_k) = 3a+a_9+a_{10} \geq 3a+a_0.
   \end{eqnarray*}

   {\bf Case (I-ii).} We  have $F_{1,2} \cdot L \geq a+a_0$ (as $F_{1,2} \cdot E_{9,10}=F_{1,2} \cdot F_1=2$ and $F_{1,2} \cdot E_i >0$ for all $i \in \{1,\ldots,10\}$ by \cite[Lemma 2.1]{klvan}).
   We have $F_2 \cdot L \geq a$ by Remark \ref{rem:intminime1}.  Thus, 
   \begin{eqnarray*}
     D' \cdot L & = & (F_1 + F_2+F_{1,2}) \cdot L = E_{9,10}\cdot L+F_2 \cdot L + F_{1,2} \cdot L \\
     & \geq & (a+a_9+a_{10}-a_0)+a+(a+a_0) = 3a+a_9+a_{10} \geq 3a+a_0.
   \end{eqnarray*}

   {\bf Case (II).} We have $E_{9,10} \cdot F_i \geq 1$ for all $i$ by \cite[Lemma 2.1]{klvan}, whence
   $E_{9,10} \cdot (F_1+\cdots+F_{10}) \geq 12$, as $ F_1+\cdots+F_{10}\sim 3D'$.

   If $E_{9,10} \cdot F_i=1$ for all but one index $i=i_0$, then
   $E_{9,10} \cdot F_{i_0} \geq 3$. It follows that $(E_{9,10} + F_{i_0})^2 \geq 6$ and $2=\phi(E_{9,10} + F_{i_0})=F_i \cdot (E_{9,10} + F_{i_0})$ for all $i \neq i_0$, contradicting the fact that $\phi$ can be computed by at most three different effective isotropic numerical classes (cf. \cite[Rem. 4.19]{cdgk}). 
Hence there exist two different indices $i$ such that $E_{9,10} \cdot F_i \geq 2$, and we will without loss of generality assume $E_{9,10} \cdot F_1 \geq 2$ and $E_{9,10} \cdot F_{2} \geq 2$. It follows in particular that $F_1,F_{2} \not \equiv E_i$ for $i \in \{1,\ldots,8\}$.

   We will divide the rest of the treatment in the following cases:
   \begin{itemize}
   \item[(II-i)]  $\{F_1,F_{2}\}=\{E_9,E_{10}\}$ up to numerical equivalence.
   \item[(II-ii)] $F_1 \equiv E_9$, $F_{2} \not \equiv E_{10}$ (or vice versa).
\item[(II-iii)] $F_1,F_{2} \not \equiv E_9,E_{10}$ and $F_{1,2} \equiv E_{i_0}$ for some $i_0 \in \{1,\ldots,8\}$.
   \item[(II-iv)] $F_1,F_{2} \not \equiv E_9,E_{10}$ and  $F_{1,2} \not \equiv E_i$ for $i \in \{1,\ldots,8\}$.
\end{itemize}

{\bf Case (II-i).} Assume without loss of generality that $F_1 \equiv E_9$, $F_{2} \equiv E_{10}$. Then $F_{1,2} \not \equiv E_i$ for $i \in \{1,\ldots,10\}$,
and also $F_{1,2} \not \equiv E_{9,10}$, as $D \not \equiv D'$. Thus,
$F_{1,2} \cdot E_i>0$ for $i \in \{1,\ldots,10\}$ and $F_{1,2} \cdot E_{9,10}>0$ by \cite[Lemma 2.1]{klvan}, whence
$F_{1,2} \cdot L \geq a+a_9+a_{10}$, so that
\begin{eqnarray*}
     D' \cdot L & = & (F_1 + F_2+F_{1,2}) \cdot L = E_{9}\cdot L+E_{10} \cdot L + F_{1,2} \cdot L \\
     & \geq & (a+a_0-a_9)+(a+a_0-a_{10})+(a+a_9+a_{10}) = 3a+2a_0 \geq 3a+a_0.
   \end{eqnarray*}

   {\bf Case (II-ii).} Since $F_{2} \not \equiv E_i$ for all $i \in \{1,\ldots,10\}$, we have $F_2 \cdot E_i \geq 1$ for all $i$ by \cite[Lemma 2.1]{klvan}. Moreover, $F_2 \cdot L \geq a+a_0$ (as $F_2 \cdot E_{9,10} \geq 2$). Since $F_{1,2} \cdot E_9=F_{1,2} \cdot F_1=2$, we have $F_{1,2} \not \equiv E_i$ for all $i \in \{1,\ldots,10\}$. Thus 
$F_{1,2} \geq a$ by Remark \ref{rem:intminime1}.
Hence
   \begin{eqnarray*}
     D' \cdot L & = & (F_1 + F_2+F_{1,2}) \cdot L = E_{9}\cdot L+F_2 \cdot L + F_{1,2} \cdot L \\
     & \geq & (a+a_0-a_9)+(a+a_0)+a = 3a+2a_0-a_{9} \geq 3a+a_0.
   \end{eqnarray*}

{\bf Case (II-iii).}  Since $F_1 \not \equiv E_i$ for all $i \in \{1,\ldots,10\}$, and $F_1 \cdot E_{9,10} \geq 2$ and $F_1 \cdot E_{i_0} =F_1 \cdot F_{1,2} \geq 2$, we have $F_1 \cdot L \geq a+a_0+a_{i_0}$. Similarly, $F_2 \cdot L \geq a+a_0+a_{i_0}$. Hence 
\begin{eqnarray*}
     D' \cdot L & = & (F_1 + F_2+F_{1,2}) \cdot L = F_1\cdot L+F_2 \cdot L + E_{i_0} \cdot L \\
     & \geq & (a+a_0+a_{i_0})+ (a+a_0+a_{i_0})+(a-a_{i_0}) = 3a+2a_0+a_{i_0} \geq 3a+a_0.
\end{eqnarray*}

   {\bf Case (II-iv).} Since $F_1 \not \equiv E_i$ for all $i \in \{1,\ldots,10\}$, and $F_1 \cdot E_{9,10} \geq 2$, we have $F_1 \cdot L \geq a+a_0$. Similarly, $F_2 \cdot L \geq a+a_0$. Since $F_{1,2} \not \equiv E_i$ for $i \in \{1,\ldots,8\}$, we have $F_{1,2} \geq a$ by Remark \ref{rem:intminime1}. Hence,
   \[ D' \cdot L  =  (F_1 + F_2+F_{1,2}) \cdot L \geq  (a+a_0)+(a+a_0)+a = 3a+2a_0 \geq 3a+a_0.\]

   We have therefore proved the first part of the lemma, namely that
$(F_1+\cdots+F_{10}) \cdot L \geq (E_1+\cdots+E_{10}) \cdot L$ for any isotropic $10$-sequence
$\{F_1,\ldots,F_{10}\}$.

Assume  that $(F_1+\cdots+F_{10}) \cdot L = (E_1+\cdots+E_{10}) \cdot L$ and $F_1 \cdot L \leq F_2 \cdot L \leq \cdots \leq F_{10} \cdot L$. Arguing by contradiction, assume  that there is an $n \in \{1,\ldots,9\}$ such that $E_i \cdot L = F_i \cdot L$ for all $i \in \{1,\ldots,n-1\}$ and $F_n \cdot L < E_n \cdot L$. Since by Lemma \ref{lemma:intminime1}, the numbers $E_i \cdot L$ with $i \in \{1,\ldots,8\}$ are the eight lowest intersections of $L$ with effective nonzero isotropic divisors, we must have $n=9$. In this case, we would have $F_i \cdot L < E_9 \cdot L$ for all $i \in \{1,\ldots,9\}$. Lemma \ref{lemma:intminime1} therefore yields that $\{F_1,\ldots,F_9\}=\{E_1,\ldots E_8,E_{9,10}\}$ up to numerical equivalence and that $E_{9,10} \cdot L \geq F_i \cdot L$ for all $i \in \{1,\ldots,9\}$. We may therefore without loss of generality assume $E_{9,10} \equiv F_9$.

Consider as before $D' \sim F_9+F_{10}+F_{9,10}$. Since $(F_1+\cdots+F_{10}) \cdot L = (E_1+\cdots+E_{10}) \cdot L$, $F_i \cdot L=E_i \cdot L$ for all $i \in \{1,\ldots,8\}$ and $F_9 \cdot L < E_9 \cdot L$, we must have $F_{10} \cdot L > E_{10} \cdot L=a+a_0-a_{10}$. We have $F_{9,10} \not \equiv E_i$ for $i \in \{1,\ldots,8\}$ (as $F_{9,10} \cdot E_{9,10}=F_{9,10} \cdot F_9=2$). If $F_{9,10} \equiv E_k$ for
$k=9$ or $10$, we have $F_{9,10} \cdot L =a+a_0-a_k \geq a+a_0-a_9$; if 
$F_{9,10} \not \equiv E_k$ for
$k \in \{9,10\}$, we have $F_{9,10} \cdot E_i>0$ for all $i \in \{1,\ldots,10\}$ (by \cite[Lemma 2.1]{klvan}) and $F_{9,10} \cdot E_{9,10}=F_{9,10} \cdot F_9 = 2$, whence $F_{9,10} \cdot L \geq a+a_0$. Thus, in any event $F_{9,10} \cdot L \geq a+a_0-a_9$. Hence, we get 
\begin{eqnarray*}
     D' \cdot L & = & (F_9 + F_{10}+F_{9,10}) \cdot L = E_{9,10} \cdot L+F_{10} \cdot L + F_{9,10} \cdot L \\
                & > & (a+a_9+a_{10}-a_0)+ (a+a_0-a_{10})+(a+a_0-a_9) = 3a+a_0,
                      \end{eqnarray*}
a contradiction.
 \end{proof}

At this point we recall Definition \ref{def:phivector}  from the introduction.

\begin{proposition} \label{prop:phivector0}
  Any isotropic $10$-sequence as in Proposition \ref{prop:sid} computes $\underline{\phi}(L)$. In particular, the coefficients $a_i$ therein are unique and given by
  \begin{eqnarray*}
   a_i & = & \phi_8(L)-\phi_i(L), \; \; i \in \{1,\ldots,7\},\\
   a_i & = & \frac{1}{3}\sum_{i=1}^{10}\phi_i(L)-2\phi_8(L)-\phi_i(L), \; \; i \in \{9,10\},\\
   a_0 & = & \frac{1}{3}\sum_{i=1}^{10}\phi_i(L)-3\phi_8(L).                      
 \end{eqnarray*}
\end{proposition}

\begin{proof}
  The first statement is an immediate consequence of  Lemma \ref{lemma:intminime2}. The computation of the coefficients $a_i$ in terms of the $\phi_j(L)$ is straightforward.
\end{proof}

\begin{remark}
  Although the coefficients $a_i$ are unique, the isotropic $10$-sequence in Proposition \ref{prop:sid} is not unique, not even up to numerical equivalence or permutation, and nor is the presentation \eqref{eq:scrivoL}. Take for instance any isotropic $10$-sequence $\{E_1,\ldots,E_{10}\}$ and a decomposition
\[
     L \sim a_1E_1+\cdots +a_7E_7+a_9E_9+a_9E_{9,10},
\]
where  
$a_1\geq \cdots \geq a_7$ and $a_9$ are nonnegative integers. This satisfies the conditions of Proposition \ref{prop:sid}. On the other hand, since $E_9+E_{9,10} \sim E_8+E_{8,10}$, we may also write
\[ L \sim a_1E_1+\cdots +a_7E_7+a_9E_8+a_9E_{8,10},
\]
whence also the isotropic $10$-sequence $\{E_1,\ldots,E_7,E_9,E_8,E_{10}\}$ satisfies the desired conditions. This is a permutation of the previous isotropic $10$-sequence, but we can also construct a different isotropic $10$-sequence satisfying the conditions of Proposition \ref{prop:sid} that is not a permutation: indeed, set $E'_i:=E_i$ for $i \in \{1,\ldots,7\}$, $E'_8:=E_{8,10}$, $E'_9:=E_{9,10}$
and $E'_{10}:=E_{8,9}$. Then $\{E'_1,\ldots,E'_{10}\}$ is an isotropic $10$-sequence with $E'_{9,10}=E_9$, satisfying the conditions of Proposition \ref{prop:sid}, as
\[
     L \sim a_1E'_1+\cdots +a_7E'_7+a_9E'_9+a_9E'_{9,10}.
\]
\end{remark}

We may summarize Propositions \ref{prop:sid} and \ref{prop:phivector0} in:

\begin{thm} \label{thm:uniquefund}
Any effective line bundle $L$ on an Enriques surface $S$ satisfying $L^2>0$ admits a fundamental presentation. Moreover, the coefficients in any fundamental presentation are unique.
  \end{thm}

\begin{definition} \label{def:fund2}
  The coefficients $a_i=a_i(L)$, $i \in\{0,1,\ldots,7,9,10\}$ and $\varepsilon_L$ appearing in any
  fundamental presentation of $L$ will be called {\em the fundamental coefficients of $L$}. 
  \end{definition}


As a consequence of Theorems \ref{mainthm2} and \ref{thm:uniquefund} we obtain:

\begin{thm} \label{mainthm2'}
  The irreducible components of $\E_{g}$ are precisely the loci parametrizing pairs of arithmetic genus $g$ with the same fundamental coefficients.  
\end{thm}

\begin{proof}
  Assume that $(S,L)$ and $(S',L')$ lie in the same irreducible component of $\E_g$. Then $\varepsilon_L=\varepsilon_{L'}$ and $\underline{\phi}(L)=\underline{\phi}(L')$, since the Picard group is invariant under deformation. By Proposition \ref{prop:phivector0} we get that $L$ and $L'$ have the same fundamental coefficients.

  Conversely, assume that $L$ and $L'$ have the same fundamental coefficients. Since a fundamental presentation is a particular type of simple isotropic decomposition, $(S,L)$ and $(S',L')$ admit the same simple decomposition type, whence they lie in the same irreducible component of $\E_g$ by Theorem \ref{mainthm2}. 
  \end{proof}

We will reformulate this theorem using the $\phi$-vector, leading to Theorem \ref{mainthm3}. 
We  first prove Theorem \ref{thm:phivector1}.

\begin{proof}[Proof of Theorem \ref{thm:phivector1}]
Property (a) follows by definition and property
(e) follows from  Lemma \ref{lemma:intminime1} and Proposition  \ref{prop:phivector0}. To prove (b) and (c), consider the expressions of the coefficients $a_i$ in terms of the $\phi_j=\phi_j(L)$ from Proposition \ref{prop:phivector0}.
  It follows that $\sum_{i=1}^{10}\phi_i$ must be  divisible by $3$.
  Moreover, 
one checks that the conditions \eqref{eq:condcoff'} are equivalent to $\frac{1}{3}\sum_{i=1}^{10}\phi_i \geq \phi_8+\phi_9+\phi_{10}$, which yields (c). 

To prove (d), 
write
\begin{eqnarray*}
  L^2& \hspace{-0.2cm} = & \hspace{-0.2cm} L \cdot(a_1E_1+\cdots+a_7E_7+a_9E_9+a_{10}E_{10}+a_{0}E_{9,10})  \\
  & \hspace{-0.2cm} = & \hspace{-0.2cm} a_1\phi_1+\cdots+ a_7\phi_7+a_{9}\phi_9+a_{10}\phi_{10}+a_0(a_1+\cdots +a_7+2a_9+2a_{10}),
  \end{eqnarray*}
and insert the values of $a_i$ from Proposition \ref{prop:phivector0}.

  The ``only if'' part of (f) is obvious. Conversely, assume that all
  $\phi_i$ are even.  Setting $a:=a_0+a_1+\cdots+a_7+a_9+a_{10}$, we have
  \begin{eqnarray*}
    \phi_8 & = & E_8 \cdot L = a, \\
    \phi_9 & = & E_9 \cdot L = a+a_0-a_9, \\
    \phi_{10} & = & E_{10} \cdot L = a+a_0-a_{10}. 
  \end{eqnarray*}
  Thus, $a$, $a_0-a_9$ and $a_0-a_{10}$ are all even. Moreover, by Proposition \ref{prop:phivector0}, all $a_i$ are even for $i \in \{1,\ldots,7\}$.  Since
  \[ a= a_1+\cdots+a_7+(a_9-a_0)+(a_{10}-a_0)+3a_0,\]
  we see that also $a_0$ is even, whence also $a_9$ and $a_{10}$. Hence, all coefficients $a_i$ in the fundamental presentation \eqref{eq:scrivoL} are even, and it follows that $L$ is numerically $2$-divisible. 

  Finally we prove (g). The fact that any isotropic $10$-sequence appearing in a fundamental presentation of $L$
  computes $\underline{\phi}$ follows from Proposition \ref{prop:phivector0}. Conversely, let $\{E_1,\ldots,E_{10}\}$ be any isotropic $10$-sequence
  computing $\underline{\phi}$. Define integers $a_i$, for $i \in \{0,1,\ldots,7,9,10\}$ as in Proposition \ref{prop:phivector0}. 
  Set $A:=L-a_1E_1-\cdots-a_7E_7-a_9E_9-a_{10}E_{10}-a_0E_{9,10}$.
Then one readily computes $A^2=0$ and
  $E_i \cdot A=0$ for all $i \in \{1,\ldots,10\}$. Hence $A \equiv 0$ by the Hodge index theorem. Thus, $L \equiv a_1E_1+\cdots+a_7E_7+a_9E_9+a_{10}E_{10}+a_0E_{9,10}$. If all coefficients $a_i$ are even, then $L$ is numerically $2$-divisible, whence $L \sim a_1E_1+\cdots+a_7E_7+a_9E_9+a_{10}E_{10}+a_0E_{9,10}+\varepsilon_L K_S$ by definition of $\varepsilon_L$ (cf. \eqref{eq:defeps}). If $a_{i_0}$ is odd, for some $i_0\in \{0,1,\ldots,7,9,10\}$, then substitute $E_{i_0}$ with $E_{i_0}+K_S$ and $E_8$ by $E_8+K_S$ if necessary, to make sure that $L \sim a_1E_1+\cdots+a_7E_7+a_9E_9+a_{10}E_{10}+a_0E_{9,10}$ (note that $\varepsilon_L=0$ by Remark \ref{rem:eps}).

  To prove the last statement, assume $(\phi_1,\ldots, \phi_{10})$ satisfies (a)-(c). Define integers $a_0,a_1,\ldots,a_7,a_9,a_{10}$ by
  $a_i  =  \phi_8-\phi_i$ for $i \in \{1,\ldots,7\}$,
  $a_i  =  \frac{1}{3}\sum_{i=1}^{10}\phi_i-2\phi_8-\phi_i$, for $ i \in \{9,10\}$, and 
  $a_0  =  \frac{1}{3}\sum_{i=1}^{10}\phi_i-3\phi_8$. Then these are all nonnegative, not all zero, and satisfy the conditions    \eqref{eq:condcoff} and \eqref{eq:condcoff'}. Take any isotropic $10$-sequence of effective divisors $\{E_1,\ldots,E_{10}\}$ and define $E_{9,10}:=\frac{1}{3}(E_1+\ldots+E_{10})-E_9-E_{10}$ as usual (cf. Lemma \ref{lemma:ceraprima}). Let $L \equiv a_1E_1+\cdots+a_7E_7+a_9E_9+a_{10}E_{10}+a_{9,10}E_{9,10}$. Then  $L$ satifies the desired conditions by Proposition \ref{prop:phivector0}.                 
\end{proof}

\begin{proof}[Proof of Theorem \ref{mainthm3}]
By Theorem \ref{mainthm2'}  the irreducible components of $\E_g$ are determined by the fundamental coefficients, which are by Proposition \ref{prop:phivector0} determined by the $\phi$-vector and value of $\varepsilon_L$. The possible values of   the $\phi$-vector are determined by conditions (i)-(iii) by Theorem \ref{thm:phivector1} and the value of $\varepsilon_L$ satisfies (iv) by Remark \ref{rem:eps}. The value of $g$ satisfies (v) by Theorem \ref{thm:phivector1}(d). The form of the fundamental presentation is again given by  Proposition \ref{prop:phivector0}.
\end{proof}

\begin{proof}[Proof of Proposition \ref{prop:dom}]
  From Proposition \ref{prop:phivector0} one finds that
$\widehat{\E}_{621;30,31,32,33,34,35,36,37,38,39}$  parametrizes pairs $(S,[L])$ with fundamental presentation
  (modulo $K_S$)
  \begin{equation} \label{eq:star}
L \equiv 7E_1+6E_2+5E_3+4E_4+3E_5+2E_6+E_7+3E_9+2E_{10}+4E_{9,10}.
  \end{equation}

  \begin{claim}
    The only isotropic $10$-sequence computing $\underline{\phi}(L)$ is, up to numerical equivalence, $(E_1,\ldots,E_{10})$. 
  \end{claim}

  \begin{proof}[Proof of claim]
    Let $E \not \equiv E_i$ for $i \in \{1,\ldots,10\}$ be any non-zero isotropic effective divisor. We will prove that $E \cdot L \geq 40 > E_{10} \cdot L=39$, unless $E \equiv E_{9,10}$ or $E_{8,10}$, in which case $E \cdot L=38$ or $39$, respectively. It will follow, since $E_8 \cdot L=37$, that the eight first members of any isotropic $10$-sequence computing $\underline{\phi}(L)$ are, up to numerical equivalence, $E_1,\ldots,E_{8}$, and that the only way to complete this to a $10$-sequence computing $\underline{\phi}(L)$ is with $E_9$ and $E_{10}$, which will prove the claim.

    Assume therefore that $E$ is non-zero effective isotropic, with $E \not \equiv E_i$ for $i \in \{1,\ldots,10\}$ and $E \not \equiv E_{9,10},E_{8,10}$. 
Since $E \cdot E_i >0$ for $i \in \{1,\ldots,10\}$ by \cite[Lemma 2.1]{klvan}, we have
\begin{equation} \label{eq:cp1}
  E \cdot (E_9+E_{10}+E_{9,10})=\frac{1}{3}E \cdot \left(E_1+\cdots+ E_{10}\right)>3,
  \end{equation}
  by Lemma \ref{lemma:ceraprima}. Therefore, as also $E \cdot E_{9,10}>0$ by \cite[Lemma 2.1]{klvan}, we have
\begin{eqnarray*}
  E \cdot L & \geq & E \cdot (7E_1+\cdots+E_7)+
                     2 E \cdot (E_9+E_{10}+E_{9,10}) + E \cdot E_9 + 2 E \cdot E_{9,10}
  \\
  & \geq & 
  7+6+5+4+3+2+1+ 2 \cdot 4 + 1+2 = 39,
  \end{eqnarray*}
and equality occurs if and only if $E \cdot E_i=1$ for $i \in \{1,\ldots,7,9\}$, $E \cdot E_{10}=2$ and $E \cdot E_{9,10}=1$. In this case we must have $E \cdot E_8 =2$ by \eqref{eq:cp1}. By Lemma \ref{lemma:ceraprima},  we have \[ 4+E \cdot E_{8,10} =E \cdot (E_8+E_{10}+E_{8,10}) = E \cdot (E_9+E_{10}+E_{9,10})=4,\]
whence $E \equiv E_{8,10}$ by \cite[Lemma 2.1]{klvan}, a contradiction.
  \end{proof}

Let $\underline{\phi}=(\phi_1,\ldots,\phi_{10})$ be any $10$-tuple of integers satisfying (i)-(iii) in Theorem \ref{mainthm3} and let $g$ be as in Theorem \ref{mainthm3}(v). Let $(a_0,a_1,\ldots,a_7,a_9,a_{10})$ be the associated fundamental coefficients as
in
  Proposition \ref{prop:phivector0}.  The claim implies that any member of the component $\widehat{\E}_{621;30,31,32,33,34,35,36,37,38,39}$ has a fundamental presentation  as in \eqref{eq:star} for a unique isotropic $10$-sequence $(E_1,\ldots,E_{10})$ up to numerical equivalence. Therefore,  the map $\mu_{\underline{\phi}}$  sending $(S,[L]=[7E_1+6E_2+5E_3+4E_4+3E_5+2E_6+E_7+3E_9+2E_{10}+4E_{9,10}])$ to
  $(S,[a_1E_1+a_2E_2+a_3E_3+a_4E_4+a_5E_5+a_6E_6+a_7E_7+a_9E_9+a_{10}E_{10}+a_0E_{9,10}])$  is a well-defined morphism 
$\mu_{\underline{\phi}}:\widehat{\E}_{621;30,31,32,33,34,35,36,37,38,39} 
  \longrightarrow \widehat{\E}_{g;\phi_1,\ldots,\phi_{10}}$.  Given any element $(S',[L'])$ in $\widehat{\E}_{g;\phi_1,\ldots,\phi_{10}}$, we may write
  \[ L' \equiv a_1E'_1+a_2E'_2+a_3E'_3+a_4E'_4+a_5E'_5+a_6E'_6+a_7E'_7+a_9E'_9+a_{10}E'_{10}+a_0E'_{9,10},\]
  for an isotropic $10$-sequence $\{E'_1,\ldots,E'_{10}\}$ computing $\phi(L')$, by Theorem \ref{mainthm3} (with $E'_{9,10} \equiv \frac{1}{3}\left(E'_1+\cdots+E'_{10}\right)$). Define
  \[ A= 7E'_1+6E'_2+5E'_3+4E'_4+3E'_5+2E'_6+E'_7+3E'_9+2E'_{10}+4E'_{9,10}.\]
Then $(S',[A]) \in \widehat{\E}_{621;30,31,32,33,34,35,36,37,38,39}$ by Theorem \ref{mainthm3} and 
$\mu_{\underline{\phi}}\left((S',[A])\right)=(S',[L'])$. Thus, $\mu_{\underline{\phi}}$ is surjective. 
\end{proof}

\begin{remark}
  The crucial point in the previous proof is the uniqueness of the isotropic $10$-sequence computing $\underline{\phi}$, as we otherwise would not have a well-defined map as claimed. One can prove that $\widehat{\E}_{621;30,31,32,33,34,35,36,37,38,39} $ is the component of lowest genus with such a property.
\end{remark}

\begin{question} \label{Q:same}
  By \cite[(4), (7) and Prop. 5.7]{GrHu} also the moduli space of Enriques surfaces with a level-$2$-structure $\widetilde{\mathcal{M}}^0_{En}$ enjoys the property that it dominates all irreducible components of the moduli space of numerically polarized Enriques surfaces. Are the spaces $\widehat{\E}_{621;30,31,32,33,34,35,36,37,38,39}$ and $ \widetilde{\mathcal{M}}^0_{En}$ isomorphic?
\end{question}

\begin{remark} \label{rem:border}
  Consider any irreducible component $\E'$ of $\E_g$ for a fixed value of \linebreak $(\phi_1,\ldots,\phi_{10},\varepsilon)$ with $\phi_1 \geq 3$. Then, as $\phi(L) =\phi_1$ for all $(S,H) \in \E'$ by Theorem \ref{thm:phivector1}, we have that $|L|$ is very ample for general $(S,L)$, cf. \cite[Thm. 4.6.1]{cd}. The general member of the irreducible component of the Hilbert scheme $\mathfrak{H}'$ containing the projective models $\phi_L(S)$ of $(S,L) \in \E'$ is thus a smooth Enriques surface of degree $2g-2$ in $\PP^{g-1}$. Proposition \ref{prop:sid} and Theorem \ref{thm:hilbert} show that there are smooth points in $\mathfrak{H}'$ represented by a reducible surface $\overline{X}=\overline{R} \cup_{\overline{T}} \cup \overline{P}$
  with $\overline{P}$ rational and $\overline{R}$ birational to the symmetric product of a double cover of the smooth elliptic curve $\overline{T}$, and
  $\overline{R}$ and $\overline{P}$ intersect transversally and only along $\overline{T}$. Computing the fundamental coefficients $a_i$ from the $\phi_i$ as in Proposition \ref{prop:phivector0} one can find out precisely what the surfaces
  $\overline{R}$ and $\overline{P}$ are as well as their hyperplane bundles, cf. Remark \ref{rem:chicont}. We thus have in each irreducible component of the Hilbert scheme of smooth Enriques surfaces a  concrete reducible member that we hope will find more applications in the future.
\end{remark}

We end by showing how to compute various irreducible components of $\E_g$ in some cases.
We will use the notation for the components proposed in the introduction.

\begin{example1*} This case corresponds to line bundles $L$ with $\phi(L)=1$, by Theorem \ref{thm:phivector1}. One readily checks that the only possibility given by conditions (i)-(iv) in Theorem \ref{mainthm3} for fixed $g$ is $(\phi_1,\ldots,\phi_{10})=(1,g-1,g,\ldots,g)$. In particular there is only one irreducible component of $\E_g$ with $\phi_1=1$, which we denote by $\E_{g;1,g-1,g,\ldots,g}$ with the above notation. The fundamental presentation is
$(g-1)E_1+E_2$.
  We thus retrieve \cite[Cor. 1.3 and Lemma 4.18(i)]{cdgk}.
\end{example1*}
  
\begin{example2*} This case corresponds to line bundles $L$ with $\phi(L)=2$, by Theorem \ref{thm:phivector1}. One readily checks that the only possibilities for $\underline{\phi}:=(\phi_1,\ldots,\phi_{10})$ given by conditions (i)-(iv) in Theorem \ref{mainthm3} for fixed $g$ are \begin{footnotesize}
  \[ \underline{\phi} \in \left\{\left(2,\frac{g+1}{2},\ldots,\frac{g+1}{2}, \frac{g+3}{2}\right), \left(2,\frac{g}{2},\frac{g}{2},\frac{g+2}{2},\ldots,\frac{g+2}{2}\right),
 \left(2,\frac{g-1}{2},\frac{g+3}{2},\ldots,\frac{g+3}{2}\right)\right\}. \]   
    \end{footnotesize}
    Thus, we get the following irreducible components of $\E_g$ with fundamental presentations:
    \begin{itemize}
    \item $\E_{g;2,\frac{g}{2},\frac{g}{2},\frac{g+2}{2},\ldots,\frac{g+2}{2}}$ for even $g \geq 4$; \; \; \; \; \; \; \; \; $\frac{g-2}{2}E_1+E_2+E_3$;
      \item $\E_{g;2,\frac{g+1}{2},\ldots,\frac{g+1}{2}, \frac{g+3}{2}}$ for odd $g \geq 3$;
\; \; \; \; \; \; \;  \hspace{0.2cm} $\frac{g-3}{2}E_1+E_9+E_{9,10}$;
      \item $\E_{g;2,\frac{g-1}{2},\frac{g+3}{2},\ldots,\frac{g+3}{2}}$ for $g \equiv 3 \; \mbox{mod} \; 4$, $g \geq 7$; \; \; $\frac{g-1}{2}E_1+2E_2$;
      \item $\E_{g;2,\frac{g-1}{2},\frac{g+3}{2},\ldots,\frac{g+3}{2}}^+$ for $g \equiv 1 \; \mbox{mod} \; 4$, $g \geq 5$; \; \; $\frac{g-1}{2}E_1+2E_2$;
        \item $\E_{g;2,\frac{g-1}{2},\frac{g+3}{2},\ldots,\frac{g+3}{2}}^-$
  for $g \equiv 1 \; \mbox{mod} \; 4$, $g \geq 5$; \; \; $\frac{g-1}{2}E_1+2E_2+K_S$.
\end{itemize}
Note that in the second case we may use \eqref{eq:ijk} and rewrite $\frac{g-3}{2}E_1+E_9+E_{9,10} \sim \frac{g-1}{2}E_1+E_{1,10}$. In particular, we see that the three irreducible components $\E_{5,2}^{(I)}$, $\E_{5,2}^{(II)^+}$ and $\E_{5,2}^{(II)^-}$ from the introduction are $\E_{5;2,3,\ldots,3, 4}$, $\E_{5;2,2,4,\ldots,4}^+$ and $\E_{5;2,2,4,\ldots,4}^-$, respectively. We also retrieve \cite[Cor. 1.3 and Lemma 4.18(ii)]{cdgk}. 
\end{example2*}

\begin{example3*} This case corresponds to line bundles $L$ with $\phi(L)=3$, by Theorem \ref{thm:phivector1}. Checking the possibilities for $\underline{\phi}:=(\phi_1,\ldots,\phi_{10})$ given by conditions (i)-(iv) in Theorem \ref{mainthm3} for fixed $g$, we get the following irreducible components of $\E_g$ with fundamental presentations: 
\begin{itemize}
\item $\E_{g; 3,\frac{g+3}{3},\ldots,\frac{g+3}{3}}$ \hspace{1.71cm} for $g \equiv 0 \; \mbox{mod} \; 3$, $g \geq 6$;  \hspace{0.4cm} $\frac{g-6}{3}E_1+E_9+E_{10}+E_{9,10}$;
 \item $\E_{g; 3,\frac{g}{3},\frac{g+3}{3},\frac{g+6}{3},\ldots,\frac{g+6}{3}}$ \hspace{0.93cm}for $g \equiv 0 \; \mbox{mod} \; 3$, $g \geq 9$; \hspace{0.35cm} $\frac{g-3}{3}E_1+2E_2+E_3$;
\item $\E_{g;3,\frac{g+2}{3},\frac{g+2}{3},\frac{g+2}{3},\frac{g+5}{3},\ldots,\frac{g+5}{3}}$  for $g \equiv 1 \; \mbox{mod} \; 3$, $g \geq 7$; \; \; $\frac{g-4}{3}E_1+E_2+E_3+E_4$;
\item $\E_{g;3,\frac{g-1}{3},\frac{g+8}{3},\ldots,\frac{g+8}{3}}$ \hspace{1.15cm} for $g \equiv 1 \; \mbox{mod} \; 3$, $g \geq 10$; \hspace{0.18cm}  $\frac{g-1}{3}E_1+3E_2$;
  \item $\E_{g; 3,\frac{g+1}{3},\frac{g+4}{3},\ldots,\frac{g+4}{3}, \frac{g+7}{3}}$ \hspace{0.55cm} for $g \equiv 2 \; \mbox{mod} \; 3$, $g \geq 8$; \hspace{0.37cm}  $\frac{g-5}{3}E_1+E_2+E_9+E_{9,10}$. 
\end{itemize}
Note that in the first case we may use \eqref{eq:ijk} and rewrite
$\frac{g-6}{3}E_1+E_9+E_{10}+E_{9,10} \sim \frac{g-3}{3}E_1+E_2+E_{1,2}$ and likewise in the last case we may rewrite $\frac{g-5}{3}E_1+E_2+E_9+E_{9,10}\sim \frac{g-2}{3}E_1+E_2+E_{1,10}$. In particular, we retrieve \cite[Cor. 1.3 and Lemma 4.18(iii)]{cdgk}. 
\end{example3*}

\begin{remark}
  Although fundamental presentations are of a suitable  form to prove the results in the present paper, there may be other effective isotropic decompositions that are more suitable to work with for other purposes. In particular, as we have seen in the above examples, in certain cases we may obtain simple isotropic decompositions with fewer isotropic components than the fundamental presentation using \eqref{eq:10-3} or \eqref{eq:ijk}. This is useful to find the relation between the $\phi$-vector and the results of \cite{cdgk}, where it is proved that irreducible components of the moduli spaces $\E_g$ are unirational if the members admit simple isotropic decompositions with at most $4$ nonzero coefficients, or $5$ nonzero coefficients or $7$ equal coefficients with all intersections between the isotropic components being $1$, cf. \cite[Thms. 1.1-1.2]{cdgk}. Let us see how to relate this to the $\phi$-vector.

A simple isotropic decomposition with $7$ equal coefficients and with all intersections $1$ between the isotropic components corresponds to the fundamental coefficients satisfying
  \begin{equation}
\label{eq:f6}   a_1=\cdots=a_7.
  \end{equation}
  A simple isotropic decomposition with at most $4$ nonzero coefficients
  can be obtained if the fundamental coefficients satisfy any of the following:
  \begin{eqnarray}
   \label{eq:f2}   & a_2=\cdots=a_7  =  0, & \\
   \label{eq:f4}  & a_3=\cdots=a_7  =  0, \; \; a_0=a_9,&  \\
   \label{eq:f3}  & a_4=\cdots=a_7  =  0, \; \; a_0=a_9=a_{10},& \\
   \label{eq:f5}   & a_3=\cdots=a_7  =  0, \; \; a_9+a_{10}=a_0.&  
  \end{eqnarray}
  This is clear in the first case; in the last three cases one may rewrite
  $a_9E_9+a_{10}E_{10}+a_0E_{9,10}$ to  $a_0E_1+a_{10}E_{10}+a_0E_{1,10}$, $a_0(E_1+E_2+E_{1,2})$,  and $a_0E_1+(a_0-a_{10})E_{9,10}+a_{10}E_{1,9}$, respectively.
  Finally,  a simple isotropic decomposition with at most $5$ nonzero coefficients with all intersections $1$ between the isotropic components corresponds to
  \begin{equation}
   \label{eq:f1}  a_6=a_7=a_9=a_{10}  = 0. 
  \end{equation}
Thus, the cases \eqref{eq:f6}-\eqref{eq:f1} all give unirational components of $\E_g$. Translating these conditions into conditions on the $\phi$-vector using 
Proposition \ref{prop:phivector0}, we obtain that {\it $\E_{g,\phi_1,\ldots,\phi_{10}}$ and $\E^{\pm}_{g,\phi_1,\ldots,\phi_{10}}$ are unirational in any of the following cases:}
  \begin{eqnarray*}
    & \phi_1=\cdots=\phi_7, & \\
    & \phi_2=\cdots=\phi_8, & \\
    & \phi_3=\cdots=\phi_9, & \\
    & \phi_4=\cdots=\phi_{10}, & \\
    & \phi_3=\cdots=\phi_8=\frac{1}{3}\left( 2(\phi_9+\phi_{10})-\phi_1-\phi_2\right), & \\
    & \phi_6=\cdots=\phi_{10}=\frac{1}{4}\left(\phi_1+\cdots+\phi_5\right). & \\
  \end{eqnarray*}
This may suggest that symmetries between the $\phi_i$ guarantee unirationality. One may obtain similar conditions for uniruledness of components, cf. again  \cite[Thms. 1.1-1.2]{cdgk}.
\end{remark}

%
%

\end{document}